\documentclass[10pt]{amsart}

\usepackage{amssymb,amsmath}
\usepackage[dvips,ps,all]{xy}

\newlength{\asidelength}

\newcommand{\Mark}{\ifinner\relax\else\marginpar{\hskip.5in$\bullet$}\fi}
\newcommand{\DashFill}{\leaders\hbox to 1ex{\hss$-$\hss}\hfill}

\newcommand{\OneStrike}[1]{\setbox0\hbox{#1}\hbox to \wd0{\hbox to0pt{#1\hss}\DashFill}}

\def\DoNextStrike#1 {\OneStrike{#1}\leavevmode{}\Delim\NextStrike}
\def\SkipNextStrike#1\EndStrike{\relax}

\numberwithin{equation}{section}
\numberwithin{subsection}{section}
\theoremstyle{plain}
\newtheorem{thm}[equation]{Theorem}

\newtheorem{prop}[equation]{Proposition}
\newtheorem{assumption}[equation]{Basic Assumption}

\theoremstyle{definition}
\newtheorem{defn}[equation]{Definition}

\theoremstyle{remark}
\newtheorem{rem}[equation]{Remark}

\newtheorem{example}[equation]{Example}

\newcommand{\sSet}{{ \mathsf{sSet} }}
\newcommand{\Chaincx}{{ \mathsf{Ch} }}

\newcommand{\Spectra}{{ \mathsf{Sp}^\Sigma }}

\newcommand{\C}{{ \mathsf{C} }}
\newcommand{\D}{{ \mathsf{D} }}
\newcommand{\E}{{ \mathsf{E} }}
\newcommand{\SymArray}{{ \mathsf{SymArray} }}
\newcommand{\SymSeq}{{ \mathsf{SymSeq} }}
\newcommand{\Array}{{ \mathsf{Array} }}
\newcommand{\Seq}{{ \mathsf{Seq} }}
\newcommand{\F}{{ \mathsf{F} }}
\newcommand{\OrdF}{{ \tilde{\mathsf{F}} }}
\newcommand{\OrdSet}{{ \mathsf{OrdSet} }}
\newcommand{\ISO}{{ \mathsf{Iso} }}

\newcommand{\Set}{{ \mathsf{Set} }}

\newcommand{\Ar}{{ \mathsf{Ar} }}

\newcommand{\Rt}{{ \mathsf{Rt} }}
\newcommand{\Lt}{{ \mathsf{Lt} }}

\newcommand{\unit}{{ \mathsf{k} }}

\newcommand{\Alg}{{ \mathsf{Alg} }}
\newcommand{\Bi}{{ \mathsf{Bi} }}
\newcommand{\Cube}{{ \mathsf{Cube} }}
\newcommand{\pCube}{{ \mathsf{pCube} }}

\newcommand{\AlgO}{{ \Alg_\capO }}
\newcommand{\LtO}{{ \Lt_\capO }}
\newcommand{\RtO}{{ \Rt_\capO }}

\newcommand{\capO}{{ \mathcal{O} }}

\newcommand{\inmap}{{ \mathrm{in} }}
\newcommand{\id}{{ \mathrm{id} }}
\newcommand{\op}{{ \mathrm{op} }}
\newcommand{\pr}{{ \mathrm{pr} }}

\newcommand{\ol}[1]{{ \overline{{#1}} }}
\newcommand{\tensor}{{ \otimes }}
\newcommand{\tensorhat}{{ \hat{\tensor} }}
\newcommand{\existsunique}{{ \exists ! }}

\newcommand{\iso}{{ \cong }}
\newcommand{\Iso}{{  \ \cong \ }}

\newcommand{\rarrow}{{ \longrightarrow }}

\newcommand{\circhat}{{ \,\hat{\circ}\, }}

\newcommand{\functor}[3]{{ {#1}:{#2}\rarrow{#3} }}
\newcommand{\function}[3]{{ {#1}:{#2}\rarrow{#3} }}

\newcommand{\subsetof}{{ \ \subset\ }}

\newdir{ >}{{}*!/-8pt/@{>}}
\newcommand{\monic}{ \ar@{ >->} }

\DeclareMathOperator*{\catend}{end}
\DeclareMathOperator*{\coend}{coend}
\DeclareMathOperator*{\colim}{colim}

\DeclareMathOperator{\Map}{Map}

\DeclareMathOperator{\U}{U}
\DeclareMathOperator{\Arrow}{Ar}

\title[Homotopy theory of modules over operads]{Homotopy theory of modules over operads and non-$\Sigma$ operads in monoidal model categories}

\author{John E. Harper}

\address{Institut de G\'eom\'etrie, Alg\`ebre et Topologie, EPFL, CH-1015 Lausanne, Switzerland}

\address{Department of Mathematics, University of Notre Dame, Notre Dame, IN 46556, USA}

\email{john.edward.harper@gmail.com}

\begin{document}
\maketitle

% Remark: (on parallel arrow spacing)
% Shifting each arrow by 0.5ex (ex is a dimension)
% seems to give good arrow separation for parallel arrows.
% Shifting each arrow by 2pt (pt is a dimension) is also OK,
% although it looks a little tight.
%
% Here is an adjunction example and a coequalizer example:
%
%\begin{align*}
%\xymatrix{
%   \SymSeq\ar@<0.5ex>[r] & \SymSeq'\ar@<0.5ex>[l]\\
%   \SymSeq\ar@<2pt>[r] & \SymSeq'\ar@<2pt>[l].
%}
%\end{align*}
%
%\begin{align*}
%\xymatrix{
%  \colim & A\ar[l] & B\ar@<-0.5ex>[l]\ar@<0.5ex>[l]\\
%  \colim & A\ar[l] & B\ar@<-2pt>[l]\ar@<2pt>[l]
%}
%\end{align*}

\section{Introduction}\label{sec:introduction}

There are many interesting situations in which algebraic structure can be described by operads \cite{Basterra_Mandell, Ginzburg_Kapranov, Goerss_Hopkins, Kriz_May, Mandell, May, McClure_Smith_conjecture}. Let $(\C,\tensor,\unit)$ be a closed symmetric monoidal category (Section~\ref{sec:closed_sym_monoidal}) with all small limits and colimits. It is possible to define two types of operads (Definition~\ref{defn:operads}) in this setting, as well as algebras and modules over these operads. One type, called $\Sigma$-operad, is based on finite sets and incorporates symmetric group actions; the other type, called non-$\Sigma$ operad, is based on ordered sets and has no symmetric group contribution. (In this paper we use the term $\Omega$-operad for non-$\Sigma$ operad, where $\Omega=\text{O}$ for ``Ordered.'') Given an operad $\capO$, we are interested in the possibility of doing homotopy theory in the categories of $\capO$-algebras and $\capO$-modules, which in practice means putting a Quillen model structure on these categories of algebras and modules. In this setting, $\capO$-algebras are the same as left $\capO$-modules concentrated at $0$ (Section~\ref{sec:algebras_over_operads}).

Of course, to get started we need some kind of homotopy theoretic structure on $\C$ itself; this structure should mesh appropriately with the monoidal structure on $\C$. The basic assumption is this.

\begin{assumption}\label{BasicAssumption}
From now on in this paper we assume that $(\C,\tensor,\unit)$ is a closed symmetric monoidal category with all small limits and colimits, that $\C$ is a cofibrantly generated model category in which the generating cofibrations and acyclic cofibrations have small domains, and that with respect to this model structure $(\C,\tensor,\unit)$ is a monoidal model category.
\end{assumption}

Model categories provide a setting in which one can do homotopy theory, and in particular, provide a framework for constructing and calculating derived functors. The extra structure of a cofibrantly generated model category is described in \cite[Definition 2.2]{Schwede_Shipley}. When we refer to the extra structure of a monoidal model category, we are using \cite[Definition 3.1]{Schwede_Shipley}; an additional condition involving the unit is assumed in \cite[Definition 2.3]{Lewis_Mandell} which we will not require in this paper. A useful introduction to model categories is given in \cite{Dwyer_Spalinski}. See also the original articles \cite{Quillen, Quillen_rational} and the more recent \cite{Chacholski_Scherer, Goerss_Jardine, Hirschhorn, Hovey}. 

The main theorem for non-$\Sigma$ operads is this.

\begin{thm}\label{MainTheorem}
Let $\capO$ be an $\Omega$-operad in $\C$. Assume that $\C$ satisfies Basic Assumption~\ref{BasicAssumption} and in addition satisfies the monoid axiom (Definition~\ref{def:monoid_axiom}). Then the category of $\capO$-algebras and the category of left $\capO$-modules both have natural model category structures. The weak equivalences and fibrations in these model structures are inherited in an appropriate sense from the weak equivalences and fibrations in~$\C$.
\end{thm}

\begin{rem}\label{rem:retracts}
Given any $\Omega$-operad $\capO$, there is an associated $\Sigma$-operad $\capO\cdot \Sigma$, such that algebras over $\capO\cdot\Sigma$ are the same as algebras over $\capO$. It follows easily from the above theorem that if $\capO'$ is a $\Sigma$-operad which is a retract of $\capO\cdot\Sigma$, then the category of algebras over $\capO'$ has a natural model category structure. 
\end{rem}

The above remark shows how to handle algebras over certain $\Sigma$-operads. We can do a lot better if $\C$ satisfies a strong cofibrancy condition. The main theorem for $\Sigma$-operads is this.

\begin{thm}\label{MainTheorem2}
Let $\capO$ be a $\Sigma$-operad in $\C$. Assume that $\C$ satisfies Basic Assumption~\ref{BasicAssumption} and in addition that every symmetric sequence (resp. symmetric array) in $\C$ is cofibrant in the projective model structure. Then the category of $\capO$-algebras (resp. left $\capO$-modules) has a natural model category structure. The weak equivalences and fibrations in these model structures are inherited in an appropriate sense from the weak equivalences and fibrations in $\C$.
\end{thm}

In setting up the machinery for Theorems \ref{MainTheorem} and \ref{MainTheorem2}, we work with projective model structures on the diagram category of \emph{(symmetric) sequences} in $\C$ (Definition~\ref{def:symmetric_sequence}) and on the diagram category of \emph{(symmetric) arrays} in $\C$ (Definition~\ref{def:symmetric_array}).

\subsection{Some examples of interest}
The hypotheses of these theorems may seem restrictive, but in fact they allow, especially in the case of Theorem~\ref{MainTheorem}, for many interesting examples including the case $(\sSet,\times,*)$ of simplicial sets \cite{Dwyer_Henn, Gabriel_Zisman, Goerss_Jardine, May_simplicial}, the case $(\Chaincx_\unit,\tensor,\unit)$ of unbounded chain complexes over a commutative ring with unit \cite{Hovey, MacLane_homology}, and the case $(\Spectra,\wedge,S)$ of symmetric spectra \cite{Hovey_Shipley_Smith, Schwede_book_project}. In a related paper \cite{Harper_Spectra}, we improve Theorem \ref{MainTheorem} to $\Sigma$-operads for the case $(\Spectra,\wedge,S)$ of symmetric spectra.

\subsection{Relationship to previous work}

One of the theorems of Schwede and Shipley~\cite{Schwede_Shipley} is that the category of monoids in $(\C,\tensor,\unit)$ has a natural model category structure, provided that the \emph{monoid axiom} (Definition~\ref{def:monoid_axiom}) is satisfied. Theorem~\ref{MainTheorem} improves this result to algebras and left modules over any $\Omega$-operad.

One of the theorems of Hinich~\cite{Hinich} is that for unbounded chain complexes over a field of characteristic zero, the category of algebras over any $\Sigma$-operad has a natural model category structure. Theorem~\ref{MainTheorem2} improves this result to the category of left modules, and provides a simplified conceptual proof of Hinich's original result. In this rational case our theorem is this.

\begin{thm}\label{thm:chain_complexes_char_zero}
Let $\unit$ be a field of characteristic zero and let $(\Chaincx_\unit,\tensor,\unit)$ be the closed symmetric monoidal category of unbounded chain complexes over $\unit$. Let $\capO$ be any $\Sigma$-operad or $\Omega$-operad. Then the category of $\capO$-algebras (resp. left $\capO$-modules) has a natural model category structure.  The weak equivalences are the homology isomorphisms (resp. objectwise homology isomorphisms) and the fibrations are the dimensionwise surjections (resp. objectwise dimensionwise surjections).
\end{thm}

\begin{rem}
Let $G$ be a finite group. If $\unit$ is a field of characteristic zero, then every $\unit[G]$-module is projective. It follows that every symmetric sequence and every symmetric array in $\Chaincx_\unit$ is cofibrant.
\end{rem}

Another theorem of Hinich \cite{Hinich} is that for unbounded chain complexes over a commutative ring with unit, the category of algebras over any $\Sigma$-operad of the form $\capO\cdot\Sigma$ for some $\Omega$-operad $\capO$, has a natural model category structure. Theorem~\ref{MainTheorem} improves this result to the category of left modules. Our theorem is this.

\begin{thm}\label{thm:chain_complexes_ring}
Let $\unit$ be a commutative ring with unit and let $(\Chaincx_\unit,\tensor,\unit)$ be the closed symmetric monoidal category of unbounded chain complexes over $\unit$. Let $\capO$ be any $\Omega$-operad. Then the category of $\capO$-algebras (resp. left $\capO$-modules) has a natural model category structure.  The weak equivalences are the homology isomorphisms (resp. objectwise homology isomorphisms) and the fibrations are the dimensionwise surjections (resp. objectwise dimensionwise surjections).
\end{thm}

One of the theorems of Elmendorf and Mandell~\cite{Elmendorf_Mandell} is that the category of simplicial multifunctors from a small multicategory (enriched over simplicial sets) to the category of symmetric spectra has a natural simplicial model category structure. Their proof involves a filtration in the underlying category of certain pushouts of algebras. We have benefitted from their paper and our proofs of Theorems \ref{MainTheorem} and \ref{MainTheorem2} exploit similar filtrations (Section~\ref{sec:proofs_for_modules}).

The framework presented in this paper for doing homotopy theory in the categories of algebras and modules over an operad is largely influenced by Rezk \cite{Rezk}.

\subsection*{Acknowledgments}

The author would like to thank Bill Dwyer for his constant encouragement and invaluable help and advice. The author is grateful to Emmanuel Farjoun for a stimulating and enjoyable visit to Hebrew University of Jerusalem in spring 2006 and for his invitation which made this possible, and to Paul Goerss and Mike Mandell for helpful comments and suggestions at a Midwest Topology Seminar. Parts of this paper were completed while the author was a visiting researcher at the Thematic Program on Geometric Applications of Homotopy Theory at the Fields Institute for Mathematics, Toronto.

\section{Preliminaries on group actions}
\label{sec:preliminaries}
Here, for reference purposes, we collect certain basic properties of group actions and adjunctions involving group actions. Some of the statements we organize into propositions. Their proofs are exercises left to the reader.

\subsection{Closed symmetric monoidal categories}
\label{sec:closed_sym_monoidal}
By Basic Assumption~\ref{BasicAssumption}, $(\C,\tensor,\unit)$ is a closed symmetric monoidal category with all small limits and colimits. In particular, $\C$ has an initial object $\emptyset$ and a terminal object $*$. For a useful introduction to monoidal categories see \cite[VII]{MacLane_categories}, followed by \cite[VII.7]{MacLane_categories} for symmetric monoidal categories. By \emph{closed} we mean there exists a functor 
\begin{align*}
  \C^\op\times\C\rarrow \C, \quad\quad (Y,Z)\longmapsto \Map(Y,Z),
\end{align*}
which we call \emph{mapping object} (or cotensor object), which fits into isomorphisms 
\begin{align}
\label{eq:adjunction_C}
   \hom_\C(X\tensor Y,Z)\Iso \hom_\C(X,\Map(Y,Z)),
\end{align}
natural in $X,Y,Z$. 

\begin{rem}
This condition is stronger than only requiring each functor $\functor{-\tensor Y}{\C}{\C}$ to have a specified right adjoint $\functor{\Map(Y,-)}{\C}{\C}$.
\end{rem}

\subsection{Group actions and $G$-objects}

The closed symmetric monoidal structure on $\C$ induces a corresponding structure on certain diagram categories.

\begin{defn}
Let $G$ be a finite group. $\C^{G^\op}$ is the category with objects the functors $\functor{X}{G^\op}{\C}$ and morphisms their natural transformations. $\C^G$ is the category with objects the functors $\functor{X}{G}{\C}$ and morphisms their natural transformations. 
\end{defn}

The diagram category $\C^{G^\op}$ (resp. $\C^G$) is isomorphic to the category of objects in $\C$ with a specified right action of $G$ (resp. left action of $G$).

\begin{prop}\label{prop:group_shaped}
Let $G$ be a finite group. Then $(\C^{G^\op},\tensor,\unit)$ has a closed symmetric monoidal structure induced from the closed symmetric monoidal structure on $(\C,\tensor,\unit)$. In particular, there are isomorphisms
\begin{align*}
  \hom_{\C^{G^\op}}(X\tensor Y,Z)\Iso\hom_{\C^{G^\op}}(X,\Map(Y,Z))
\end{align*}
natural in $X,Y,Z$.
\end{prop}

The proposition remains true when $\C^{G^\op}$ is replaced by $\C^G$. We usually leave such corresponding statements to the reader. 

\subsection{Copowers}

If $X$ is a finite set, denote by $|X|$ the number of elements in $X$. Let $X$ be a finite set and $A\in\C$. Recall from \cite[III.3]{MacLane_categories} the \emph{copower} $A\cdot X\in\C$ is defined by:
\begin{align*}
  A\cdot X:=\coprod\limits_{X}A,
\end{align*}
the coproduct in $\C$ of $|X|$ copies of $A$.

\begin{rem}
In the literature, copower is sometimes indicated by a tensor product symbol, but because several tensor products already appear in this paper, we are using the usual dot notation as in \cite{MacLane_categories}.
\end{rem}

Let $G$ be a finite group. When $\C=\sSet$ there are natural isomorphisms $A\cdot G\Iso A\times G$, when $\C=\Chaincx_\unit$ there are natural isomorphisms $A\cdot G\Iso A\tensor \unit[G]$, and when $\C=\Spectra$ there are natural isomorphisms $A\cdot G\Iso A\wedge G_{+}$. Since left Kan extensions may be calculated objectwise in terms of copowers \cite[X.4]{MacLane_categories}, the copower construction appears in several adjunctions below.

\subsection{$G$-orbits, $G$-fixed points, and related adjunctions}

\begin{defn}
Let $G$ be a finite group. If $\functor{Y}{G^\op\times G}{\C}$ and $\functor{Z}{G\times G^\op}{\C}$ are functors, then $Y_G\in\C$ and $Z^G\in\C$ are defined by
\begin{align*}
  Y_G := \coend Y,\quad\quad Z^G:=\catend Z.
\end{align*}
\end{defn}

The universal properties satisfied by these coends and ends are convenient when working with $Y_G$ and $Z^G$, but the reader may take the following calculations as definitions. There are natural isomorphisms,
\begin{align*}
  Y_G&\Iso\colim(
  \xymatrix{
  G^\op\ar[r]^-{\mathrm{diag}} & 
  G^\op\times G^\op\Iso G^\op\times G\ar[r]^-{Y} & \C
  }),\\
  Z^G&\Iso\lim(
  \xymatrix{
  G^\op\ar[r]^-{\mathrm{diag}} & 
  G^\op\times G^\op\Iso G\times G^\op\ar[r]^-{Z} & \C
  }).
\end{align*}

\begin{prop}\label{prop:diagram_adjunctions}
Let $G$ be a finite group, $H\subset G$ a subgroup, and $\function{l}{H}{G}$ the inclusion of groups. Let $G_1,G_2$ be finite groups and $A_2\in\C^{G_2^\op}$. There are adjunctions
\begin{align}\label{eq:useful_adjunctions_diagrams}
\xymatrix{
  \C\ar@<0.5ex>[r] & \C^{H^\op}\ar@<0.5ex>[l]^-{\lim}
  \ar@<0.5ex>[r]^{-\cdot_H G} & 
  \C^{G^\op}\ar@<0.5ex>[l]^-{l^*}
}, \quad\quad
\xymatrix{
   \C^{G_1^\op}\ar@<0.5ex>[rr]^-{-\tensor A_2}&&
   \C^{(G_1\times G_2)^\op}\ar@<0.5ex>[ll]^-{\Map(A_2,-)^{G_2}}
},
\end{align}
with left adjoints on top. In particular, there are isomorphisms
\begin{align*}
  \hom_{\C^{G^\op}}(A\cdot_H G,B)
  &\Iso\hom_{\C^{H^\op}}(A,B),\\
  \hom_{\C^{G^\op}}(A\cdot(H\backslash G),B)
  &\Iso\hom_{\C}(A,B^H),\\
  \hom_{\C^{G^\op}}(A\cdot G,B)
  &\Iso\hom_{\C}(A,B),\\
  \hom_{\C^{(G_1\times G_2)^\op}}
  (A_1\tensor A_2,X)&\Iso\hom_{\C^{G_1^\op}}(A_1,\Map(A_2,X)^{G_2}),
\end{align*}
natural in $A,B$ and $A_1,X$.
\end{prop}
\begin{rem}
The restriction functor $l^*$ is sometimes dropped from the notation, as in the natural isomorphisms in Proposition~\ref{prop:diagram_adjunctions}.
\end{rem}

\section{Sequences and symmetric sequences}
\label{sec:symmetric_sequences}

In preparation for defining operads, we consider sequences and symmetric sequences of objects in $\C$. We introduce a symmetric monoidal structure $\tensor$ on $\SymSeq$, and a symmetric monoidal structure $\hat{\tensor}$ on $\Seq$. Both of these are relatively simple; $\tensor$ is a form of the symmetric monoidal structure that is used in the construction of symmetric spectra \cite{Hovey_spectra, Hovey_Shipley_Smith}, while $\hat{\tensor}$ is defined in a way that is very similar to the definition of the graded tensor product of chain complexes.
These monoidal products possess appropriate adjoints, which can be interpreted as mapping objects. For instance, there are objects $\Map^\tensor(B,C)$ and $\Map^{\hat{\tensor}}(Y,Z)$ which fit into isomorphisms
\begin{align*}
  \hom(A\tensor B,C)&\Iso\hom(A,\Map^\tensor(B,C)),\\
  \hom(X\hat{\tensor}Y,Z)&\Iso\hom(X,\Map^{\hat{\tensor}}(Y,Z)),
\end{align*}
natural in the symmetric sequences $A,B,C$ and the sequences $X,Y,Z$. The material in this section is largely influenced by \cite{Rezk}.

\subsection{Sequences and symmetric sequences}
Define the sets $\mathbf{n}:=\{1,\dots,n\}$ for each $n\geq 0$, where $\mathbf{0}:=\emptyset$ denotes the empty set. When regarded as a totally ordered set, $\mathbf{n}$ is given its natural ordering.

\begin{defn}\label{def:symmetric_sequence}
Let $n\geq 0$.
\begin{itemize}
\item $\Sigma$ is the category of finite sets and their bijections. $\Omega$ is the category of totally ordered finite sets and their order preserving bijections.
\item A \emph{symmetric sequence} in $\C$ is a functor $\functor{A}{\Sigma^{\op}}{\C}$. A \emph{sequence} in $\C$ is a functor $\functor{X}{\Omega^\op}{\C}$. $\SymSeq:=\C^{\Sigma^{\op}}$ is the category of symmetric sequences in $\C$ and their natural transformations. $\Seq:=\C^{\Omega^{\op}}$ is the category of sequences in $\C$ and their natural transformations.
\item A (symmetric) sequence $A$ is \emph{concentrated at $n$} if $A[\mathbf{s}]=\emptyset$ for all $s\neq n$; the initial object in $\C$ is denoted by $\emptyset$.
\end{itemize}
\end{defn}

\subsection{Small skeletons}
The indexing categories for symmetric sequences and sequences are not small, but they have small skeletons, which will be useful for calculations.

\begin{defn}\quad
\begin{itemize}
\item $\Sigma_n$ is the category with exactly one object $\mathbf{n}$ and morphisms the bijections of sets. $\Omega_n$ is the category with exactly one object $\mathbf{n}$ and morphisms the identity map.
\item $\Sigma'\subsetof\Sigma$ is the subcategory with objects the sets $\mathbf{n}$ for $n\geq 0$ and morphisms the bijections of sets. $\Omega'\subsetof\Omega$ is the subcategory with objects the totally ordered sets $\mathbf{n}$ for $n\geq 0$ and morphisms the identity maps.
\end{itemize}
\end{defn}

Note that $\Sigma'$ is a small skeleton of $\Sigma$ and $\Omega'$ is a small skeleton of $\Omega$.

\subsection{Tensor products for (symmetric) sequences}

Sequences and symmetric sequences have naturally occuring tensor products.

\begin{defn}\label{def:tensor_products}
Let $A_1,\dotsc,A_t$ be symmetric sequences and let $X_1,\dotsc,X_t$ be sequences. The \emph{tensor products} $A_1\tensor\dotsb\tensor A_t\in\SymSeq$ and $X_1\hat{\tensor}\dotsb\hat{\tensor}X_t\in\Seq$ are the left Kan extensions of objectwise tensor along coproduct of sets,
\begin{align*}
\xymatrix{
  (\Sigma^{\op})^{\times t}
  \ar[rr]^-{A_1\times\dotsb\times A_t}\ar[d]^{\coprod} & &
  \C^{\times t}\ar[r]^-{\tensor} & \C \\
  \Sigma^{\op}\ar[rrr]^{A_1\tensor\dotsb\tensor 
  A_t}_{\text{left Kan extension}} & & & \C,
} &&
\xymatrix{
  (\Omega^{\op})^{\times t}
  \ar[rr]^-{X_1\times\dotsb\times X_t}\ar[d]^{\coprod} & &
  \C^{\times t}\ar[r]^-{\tensor} & \C \\
  \Omega^{\op}\ar[rrr]^{X_1\hat{\tensor}\dotsb\hat{\tensor} 
  X_t}_{\text{left Kan extension}} & & & \C.
}
\end{align*}
\end{defn}

This construction is where the distinction between $\tensor$ and $\tensorhat$ really arises. The following calculation is an exercise left to the reader.

\begin{prop}\label{prop:tensor_calculation}
Let $A_1,\dotsc,A_t$ be symmetric sequences and $S\in\Sigma$, with $s:=|S|$. Let $X_1,\dotsc,X_t$ be sequences and $M\in\Omega$, with $m:=|M|$. There are natural isomorphisms,
\begin{align}
  (A_1\tensor\dotsb\tensor A_t)[S]&\Iso\ 
  \coprod_{\substack{\function{\pi}{S}{\mathbf{t}}\\ \text{in $\Set$}}}
  A_1[\pi^{-1}(1)]\tensor\dotsb\tensor
  A_t[\pi^{-1}(t)],\label{eq:tensors0}\\
  &\Iso
  \coprod_{s_1+\dotsb +s_t=s}A_1[\mathbf{s_1}]\tensor\dotsb\tensor 
  A_t[\mathbf{s_t}]\underset{{\Sigma_{s_1}\times\dotsb\times
  \Sigma_{s_t}}}{\cdot}\Sigma_{s},\label{eq:tensors1}\\
  (X_1\hat{\tensor}\dotsb\hat{\tensor} X_t)[M]&\Iso\ 
  \coprod_{\substack{\function{\pi}{M}{\mathbf{t}}\\ \text{in $\OrdSet$}}}
  X_1[\pi^{-1}(1)]\tensor\dotsb\tensor
  X_t[\pi^{-1}(t)],\label{eq:tensors2}\\
  \label{eq:tensors3}
  &\Iso
  \coprod_{m_1+\dotsb +m_t=m}X_1[\mathbf{m_1}]\tensor\dotsb\tensor 
  X_t[\mathbf{m_t}],
\end{align}
\end{prop}

\begin{rem}
Giving a map of sets $\function{\pi}{S}{\mathbf{t}}$ is the same as giving an ordered partition $(I_1,\dotsc,I_t)$ of $S$. Whenever $\pi$ is not surjective, at least one $I_j$ will be the empty set $\mathbf{0}$.
\end{rem}

\subsection{Tensor powers}

It will be useful to extend the definition of tensor powers $A^{\tensor t}$ and $X^{\hat{\tensor}n}$ to situations in which the integers $t$ and $n$ are replaced, respectively, by a finite set $T$ or a finite ordered set $N$. The calculations in Proposition~\ref{prop:tensor_calculation} suggest how to proceed. We introduce here the suggestive bracket notation used in \cite{Rezk}.

\begin{defn} \label{def:tensor_powers}
Let $A$ be a symmetric sequence and $S,T\in\Sigma$. Let $X$ be a sequence and $M,N\in\Omega$.  The \emph{tensor powers} $A^{\tensor T}\in\SymSeq$ and $X^{\hat{\tensor} N}\in\Seq$ are defined objectwise by
\begin{align}
  \label{eq:tensor_powers}
  (A^{\tensor T})[S]:= A[S,T]:=
  &\coprod_{\substack{\function{\pi}{S}{T}\\ \text{in $\Set$}}}
  \tensor_{t\in T} A[\pi^{-1}(t)],\quad T\neq \emptyset\, ,\\
  \notag
  (A^{\tensor \emptyset})[S]:= A[S,\emptyset]:=
  &\coprod_{\substack{\function{\pi}{S}{\emptyset}\\ \text{in $\Set$}}}
  \unit,\\
  \notag
  (X^{\hat{\tensor} N})[M]:=X\langle M,N\rangle:=
  &\coprod_{\substack{\function{\pi}{M}{N}\\ \text{in $\OrdSet$}}}
  \tensor_{n\in N} X[\pi^{-1}(n)],\quad N\neq \emptyset,\\
  \notag
  (X^{\hat{\tensor} \emptyset})[M]:=X\langle M,\emptyset\rangle:=
  &\coprod_{\substack{\function{\pi}{M}{\emptyset}\\ \text{in $\OrdSet$}}}
  \unit.
\end{align}
Note that there are no functions $\function{\pi}{S}{\emptyset}$ in $\Set$  unless $S=\emptyset$, and similarly with $S$ replaced by $M$. We will use the abbreviations $A^{\tensor 0}:=A^{\tensor\emptyset}$ and $X^{\hat{\tensor} 0}:=X^{\hat{\tensor}\emptyset}$.
\end{defn}

\begin{rem}
We denote by $\Set$ the category of sets and their maps, and by $\OrdSet$ the category of totally ordered sets and their order preserving maps.
\end{rem}

The above constructions give functors
\begin{align*}
  \SymSeq\times\Sigma^{\op}\times\Sigma\rarrow\C,&\quad\quad
  (A,S,T)\longmapsto A[S,T],\\
  \Seq\times\Omega^{\op}\times\Omega\rarrow\C,&\quad\quad
  (X,M,N)\longmapsto X\langle M,N\rangle,\\
  \SymSeq\times\SymSeq\rarrow\SymSeq,&\quad\quad
  (A,B)\longmapsto A\tensor B,\\
  \Seq\times\Seq\rarrow\Seq,&\quad\quad
  (X,Y)\longmapsto X\hat{\tensor} Y.
\end{align*}

Observe that the unit for the tensor product $\tensor$ on $\SymSeq$ and the unit for the tensor product $\hat{\tensor}$ on $\Seq$, both denoted ``$1$'', are given by the same formula
\begin{align*}
  1[S] :=
  \left\{
    \begin{array}{rl}
    \unit,&\text{for $|S|=0$,}\\
    \emptyset,&\text{otherwise}.
    \end{array}
    \right.
\end{align*}
Note that $A^{\tensor\emptyset}=1=M^{\tensorhat\emptyset}$. The following calculations follow directly from Definition~\ref{def:tensor_powers}; similar calculations are true for sequences.

\begin{prop} 
Let $A,B$ be symmetric sequences. There are natural isomorphisms,
\begin{align*}
  &A\tensor 1\Iso A, & &A^{\tensor 0}\Iso A[-,\mathbf{0}]\Iso 1,
  & &A\tensor B\Iso B\tensor A\\
  &A\tensor \emptyset\Iso \emptyset, & &A^{\tensor 1}\Iso 
  A[-,\mathbf{1}]\Iso A,  
  & &(A^{\tensor t})[\mathbf{0}]\Iso A[\mathbf{0},\mathbf{t}]
  \Iso A[\mathbf{0}]^{\tensor t},\quad 
  t\geq 0.
\end{align*}
Here, $\emptyset$ denotes the initial object in the category of symmetric sequences.
\end{prop}

\subsection{Mapping objects for (symmetric) sequences}

Let $B,C$ be symmetric sequences and $T\in\Sigma$. Let $Y,Z$ be sequences and $N\in\Omega$. There are functors
\begin{align*}
  \Sigma\times\Sigma^\op\rarrow\C,&\quad\quad
  (S,S')\longmapsto \Map(B[S],C[T\amalg S']),\\
  \Omega\times\Omega^\op\rarrow\C,&\quad\quad
  (M,M')\longmapsto \Map(Y[M],Z[N\amalg M']),
\end{align*}
which are useful for defining the mapping objects of $(\SymSeq,\tensor,1)$ and $(\Seq,\hat{\tensor},1)$.

\begin{defn}\label{def:map_tensor_object} Let $B,C$ be symmetric sequences and $T\in\Sigma$. Let $Y,Z$ be sequences and $N\in\Omega$. The \emph{mapping objects} $\Map^\tensor(B,C)\in\SymSeq$ and $\Map^{\hat{\tensor}}(Y,Z)\in\Seq$ are defined objectwise by the ends
\begin{align*}
  \Map^\tensor(B,C)[T]\ &:=\ \Map(B,C[T\amalg -])^\Sigma,\\
  \Map^{\hat{\tensor}}(Y,Z)[N]\ &:=\ \Map(Y,Z[N\amalg -])^\Omega.
\end{align*}
\end{defn}

Hence $\Map^\tensor(B,C)$ satisfies objectwise the universal property
\begin{align}\label{eq:mapping_tensor_universal}
\xymatrix{
   & & \Map(B[S],C[T\amalg S])
   \ar[d]^-{(\id,(\id\amalg\zeta)^*)} & S \\
   \cdot\ar@(u,l)[urr]^-{f[S]}
   \ar@(d,l)[drr]_-{f[S']}
   \ar@{.>}[r]^-{\bar{f}}_-{\existsunique}
   & \Map^\tensor(B,C)[T]\ar@/^1pc/[ur]_-{\tau[S]}
  \ar@/_1pc/[dr]^-{\tau[S']}
  & \Map(B[S],C[T\amalg S'])  \\
   & & \Map(B[S'],C[T\amalg S'])\ar[u]_-{(\zeta^*,\id)} 
   & S'\ar[uu]_{\zeta}   
}
\end{align}
that each wedge $f$ factors uniquely through the terminal wedge $\tau$ of $\Map^\tensor(B,C)[T]$. The mapping objects $\Map^{\hat{\tensor}}(Y,Z)$ satisfy objectwise a similar universal property.

\begin{prop}\label{prop:calc_map_tensor_object}
Let $B,C$ be symmetric sequences and $T\in\Sigma$, with $t:=|T|$. Let $Y,Z$ be sequences and $N\in\Omega$, with $n:=|N|$. There are natural isomorphisms,
\begin{align}
  \label{eq:mapping_tensor_calc_objectwise}
  \Map^\tensor(B,C)[\mathbf{t}]
  &\Iso\prod_{s\geq 0}\Map(B[\mathbf{s}],C[\mathbf{t}
  \boldsymbol{+}\mathbf{s}])^{\Sigma_s},\\
  \Map^{\hat{\tensor}}(Y,Z)[\mathbf{n}]
  &\Iso\prod_{m\geq 0}\Map(Y[\mathbf{m}],Z[\mathbf{n}
  \boldsymbol{+}\mathbf{m}]).
\end{align}
\end{prop}
  
\begin{proof}
Consider \eqref{eq:mapping_tensor_calc_objectwise}. Using the universal property \eqref{eq:mapping_tensor_universal} and restricting to a small skeleton $\Sigma'\subsetof\Sigma$, it is easy to obtain natural isomorphisms
\begin{align*}
  \Map^\tensor(B,C)[\mathbf{t}]&\Iso\lim
  \biggl( 
  \xymatrix{
  \prod_{\substack{S\\ \text{in}\,\Sigma'}}\limits
  \Map(B[S],C[T\amalg S])\ar@<2.5ex>[r]\ar@<1.5ex>[r] & 
  \prod_{\substack{\function{\zeta}{S'}{S}\\ \text{in}\,\Sigma'}}\limits 
  \Map(B[S],C[T\amalg S'])
  }
  \biggr).
\end{align*}
This verifies \eqref{eq:mapping_tensor_calc_objectwise} and a similar argument verifies the case for sequences.
\end{proof}

These constructions give functors
\begin{align*}
  \SymSeq^\op\times\SymSeq\rarrow\SymSeq,&\quad\quad
  (B,C)\longmapsto\Map^\tensor(B,C),\\
  \Seq^\op\times\Seq\rarrow\Seq,&\quad\quad
  (Y,Z)\longmapsto\Map^{\hat{\tensor}}(Y,Z).
\end{align*}

\begin{prop}\label{prop:adjunction_tensor_object}
Let $A,B,C$ be symmetric sequences and let $X,Y,Z$ be sequences. There are isomorphims 
\begin{align}\label{eq:adjunction_tensor_object}
  \hom(A\tensor B,C)&\Iso\hom(A,\Map^\tensor(B,C)),\\
  \label{eq:adjunction_tensor_object_bar}
  \hom(X\hat{\tensor} Y,Z)&\Iso\hom(X,\Map^{\hat{\tensor}}(Y,Z)),
\end{align}
natural in $A,B,C$ and $X,Y,Z$.
\end{prop}

\begin{proof}
Consider \eqref{eq:adjunction_tensor_object}. Using the calculation \eqref{eq:tensors0} together with the universal property \eqref{eq:mapping_tensor_universal} and the natural correspondence \eqref{eq:adjunction_C}, is it easy to verify that giving a map $A\tensor B\rarrow C$ is the same as giving a map $A\rarrow\Map^\tensor(B,C)$, and that the resulting correspondence is natural. A similar argument verifies the case for sequences.
\end{proof}

\subsection{Monoidal structures}

\begin{prop}\label{prop:tensor_monoidal}
$(\SymSeq,\tensor,1)$ and $(\Seq,\hat{\tensor},1)$ have the structure of closed symmetric monoidal categories with all small limits and colimits. 
\end{prop}

\begin{proof}
Consider the case of symmetric sequences. To verify the symmetric monoidal structure, it is easy to use \eqref{eq:tensors0} to describe the required natural isomorphisms and to verify the appropriate diagrams commute. Proposition~\ref{prop:adjunction_tensor_object} verifies the symmetric monoidal structure is closed. Limits and colimits are calculated objectwise. A similar argument verifies the case for sequences.
\end{proof}

\section{Circle products for (symmetric) sequences}
\label{sec:circle_products}

We describe a \emph{circle product} $\circ$ on $\SymSeq$ and a related circle product $\circhat$ on $\Seq$. These are monoidal products which are \emph{not} symmetric monoidal,  and they figure in the definitions of $\Sigma$-operad and $\Omega$-operad respectively (Definition ~\ref{defn:operads}). Perhaps surprisingly, these monoidal products possess appropriate adjoints, which can be interpreted as mapping objects. For instance, there are objects $\Map^\circ(B,C)$ and $\Map^{\hat{\circ}}(Y,Z)$ which fit into isomorphisms 
\begin{align*}
  \hom(A\circ B,C)&\Iso\hom(A,\Map^\circ(B,C)),\\
  \hom(X\circhat Y,Z)&\Iso\hom(X,\Map^{\hat{\circ}}(Y,Z)),
\end{align*}
natural in the symmetric sequences $A,B,C$ and the sequences $X,Y,Z$. 

The material in this section is largely influenced by \cite{Rezk}; earlier work exploiting circle product $\circ$ for symmetric sequences includes \cite{Getzler_Jones, Smirnov}, and more recent work includes \cite{Fresse_lie_theory, Fresse, Fresse_modules, Kapranov_Manin, Kelly}. The circle product $\hat{\circ}$  is used in \cite{Batanin} for working with $\Omega$-operads and their algebras.

\subsection{Circle products (or composition products)}

Let $A,B$ be symmetric sequences and $S\in\Sigma$. Let $X,Y$ be sequences and $M\in\Omega$. There are functors
\begin{align*}
  \Sigma^{\op}\times\Sigma\rarrow\C,& \quad\quad
  (T',T)\longmapsto A[T']\tensor B[S,T],\\
  \Omega^{\op}\times\Omega\rarrow\C,& \quad\quad
  (N',N)\longmapsto X[N']\tensor Y\langle M,N\rangle,
\end{align*}
which are useful for defining circle products of symmetric sequences and circle products of sequences.

\begin{defn}\label{def:circle}Let $A,B$ be symmetric sequences and $S\in\Sigma$. Let $X,Y$ be sequences and $M\in\Omega$. The \emph{circle products} (or composition products) $A\circ B\in\SymSeq$ and $X\circhat Y\in\Seq$ are defined objectwise by the coends
\begin{align*}
  (A\circ B)[S] \ &:= \ A\tensor_\Sigma (B^{\tensor-})[S] \ = \ 
  A\tensor_\Sigma B[S,-],\\
  (X\circhat Y)[M] \ &:= \ X\tensor_\Omega (Y^{\hat{\tensor}-})[M] \ = \
  X\tensor_\Omega Y\langle M,-\rangle.
\end{align*}
\end{defn}

Hence $A\circ B$ satisfies objectwise the universal property
\begin{align}\label{eq:circle_product_universal}
\xymatrix{
  T\ar[dd]^{\xi} &  
  A[T]\tensor B[S,T]\ar@/^1pc/[dr]_-{i[T]}\ar@(r,u)[rrd]^-{f[T]} & \\
    & A[T']\tensor B[S,T]\ar[u]^-{\xi^*\tensor[\id,\id]}   
    \ar[d]_-{\id\tensor[\id,\xi]}
    & (A\circ B)[S]\ar@{.>}[r]^-{\bar{f}}_-{\existsunique} & \cdot\\
  T' & A[T']\tensor B[S,T']\ar@/_1pc/[ur]^-{i[T']}\ar@(r,d)[rru]_-{f[T']} &
}
\end{align}
that each wedge $f$ factors uniquely through the initial wedge $i$ of $(A\circ B)[S]$. A similar universal property is satisfied objectwise by the circle products $X\circhat Y$. 

\begin{prop}\label{prop:circ_calc}
Let $A,B$ be symmetric sequences and $S\in\Sigma$, with $s:=|S|$. Let $X,Y$ be sequences and $M\in\Omega$, with $m:=|M|$. There are natural isomorphisms,
\begin{align}
  \label{eq:circ_calc}
  (A\circ B)[\mathbf{s}] &\Iso 
  \coprod_{t\geq 0}A[\mathbf{t}]\tensor_{\Sigma_t}
  (B^{\tensor t})[\mathbf{s}]\Iso
  \coprod_{t\geq 0}A[\mathbf{t}]\tensor_{\Sigma_t}
  B[\mathbf{s},\mathbf{t}],\\
  \notag
  (X\circhat Y)[\mathbf{m}] &\Iso
  \coprod_{n\geq 0}X[\mathbf{n}]\tensor (Y^{\hat{\tensor} n})[\mathbf{m}]\Iso
  \coprod_{n\geq 0}X[\mathbf{n}]\tensor Y\langle\mathbf{m},\mathbf{n}\rangle.
\end{align}
\end{prop}

\begin{proof}
Consider \eqref{eq:circ_calc}. Using the universal property \eqref{eq:circle_product_universal} and restricting to a small skeleton $\Sigma'\subsetof\Sigma$, it is easy to obtain natural isomorphisms
\begin{align*}
  (A\circ B)[\mathbf{s}] \Iso\colim
  \left( 
  \xymatrix{
  \coprod_{\substack{\function{\xi}{T}{T'}\\ \text{in}\,\Sigma'}}\limits
  A[T']\tensor B[S,T]\ar@<2.5ex>[r]\ar@<1.5ex>[r] & 
  \coprod_{\substack{T\\ \text{in}\,\Sigma'}}\limits A[T]\tensor B[S,T]
  }
  \right).
\end{align*}
Note that all morphisms in $\Sigma$ and $\Sigma'$ are isomorphisms. This verifies \eqref{eq:circ_calc} and a similar argument verifies the case for sequences.
\end{proof}

These constructions give functors
\begin{align*}
  \SymSeq\times\SymSeq\rarrow\SymSeq,&\quad\quad
  (A,B)\longmapsto A\circ B,\\
  \Seq\times\Seq\rarrow\Seq,&\quad\quad
  (X,Y)\longmapsto X\circhat Y.
\end{align*}
Observe that the (two-sided) unit for the circle product $\circ$ on $\SymSeq$ and the (two-sided) unit for the circle product $\hat{\circ}$ on $\Seq$, both denoted ``$I$'', are given by the same formula
\begin{align*}
  I[S] :=
  \left\{
    \begin{array}{rl}
    \unit,&\text{for $|S|=1$,}\\
    \emptyset,&\text{otherwise}.
    \end{array}
    \right.
\end{align*}

\begin{defn}
Let $A$ be a symmetric sequence, $X$ a sequence, and $Z\in\C$. The corresponding functors $\functor{A\circ(-)}{\C}{\C}$ and $\functor{X\circhat(-)}{\C}{\C}$ are defined objectwise by,
\begin{align*}
  A\circ(Z):= \coprod\limits_{t\geq 0}A[\mathbf{t}]
  \tensor_{\Sigma_t}Z^{\tensor t},\quad\quad
  X\circhat(Z):= \coprod\limits_{n\geq 0}X[\mathbf{n}]
  \tensor Z^{\tensor n}.
\end{align*}
\end{defn}

The category $\C$ embeds in $\SymSeq$ (resp. $\Seq$) as the full subcategory of symmetric sequences (resp. sequences) concentrated at $0$, via the functor $\functor{\hat{-}}{\C}{\SymSeq}$ (resp. $\functor{\hat{-}}{\C}{\Seq}$) defined objectwise by
\begin{align}\label{eq:embedding_at_zero}
  \hat{Z}[S] :=
  \left\{
    \begin{array}{rl}
    Z,&\text{for $|S|=0$,}\\
    \emptyset,&\text{otherwise}.
    \end{array}
  \right.
\end{align}

The following calculations follow directly from Proposition~\ref{prop:circ_calc}; similar calculations are true for sequences. 

\begin{prop}\label{prop:useful_calculations_for_circ}
Let $A,B$ be symmetric sequences, $s\geq 0$, and $Z\in\C$. There are natural isomorphisms,
\begin{align}
  \notag
  \emptyset\circ A &\Iso \emptyset,\quad
  I\circ A\Iso A,\quad A\circ I\Iso A,\\ 
  \notag
  (A\circ \emptyset)[\mathbf{s}]&\Iso
  \left\{
    \begin{array}{rl}
    A[\mathbf{0}],&\text{for $s=0$,}\\
    \emptyset,&\text{otherwise,}
  \end{array}
  \right.\\
  \label{eq:another_useful_hat_relation}   
  (A\circ \hat{Z})[\mathbf{s}]&\Iso
  \left\{
    \begin{array}{rl}
    A\circ(Z),
    &\text{for $s=0$,}\\
    \emptyset,&\text{otherwise},
    \end{array}
  \right.\\
  \label{eq:useful_relation_for_conc_at_zero}
  A\circ(Z)&\Iso(A\circ\hat{Z})[\mathbf{0}],\\
  \notag
  (A\circ B)[\mathbf{0}]&\Iso A\circ(B[\mathbf{0}]).
\end{align}
\end{prop}

\subsection{Properties of tensor and circle products}

It is useful to understand how tensor products and circle products interact.

\begin{prop}\label{prop:circle_and_powers}
Let $A,B,C$ be symmetric sequences, $X,Y,Z$ be sequences, and $t\geq 0$. There are natural isomorphisms
\begin{align}
  \label{eq:tensor_power_sym}
  (A\tensor B)\circ C& \Iso (A\circ C)\tensor (B\circ C),&
  (B^{\tensor t})\circ C& \Iso(B\circ C)^{\tensor t},\\
  \notag
  (X\hat{\tensor} Y)\circhat Z& \Iso (X\circhat Z)\hat{\tensor} 
  (Y\circhat Z),&
  (Y^{\hat{\tensor} t})\circhat Z& \Iso(Y\circhat Z)^{\hat{\tensor} t}.
\end{align}
\end{prop}

\begin{proof}
The case for symmetric sequences follows from \eqref{eq:tensors1} and \eqref{eq:circ_calc}, and the argument for sequences is similar.
\end{proof}

The following is a special case of particular interest. The right-hand isomorphisms may be regarded as a motivating property for the tensor products.

\begin{prop}
Let $A,B$ be symmetric sequences and let $X,Y$ be sequences. Suppose $Z\in\C$ and $t\geq 0$. There are natural isomorphisms
\begin{align*}
  (A\tensor B)\circ(Z)& \Iso (A\circ(Z))\tensor (B\circ(Z)),&
  (B^{\tensor t})\circ(Z)& \Iso (B\circ(Z))^{\tensor t},\\
  (X\hat{\tensor} Y)\circhat(Z)& \Iso (X\circhat(Z))
  \tensor(Y\circhat(Z)),&
  (Y^{\hat{\tensor} t})\circhat(Z)& \Iso (Y\circhat(Z))^{\tensor t}.
\end{align*}
\end{prop}

\begin{proof}
This follows from Proposition~\ref{prop:circle_and_powers} using the embedding \eqref{eq:embedding_at_zero}.
\end{proof}

\begin{prop}\label{prop:circ_assoc}
Let $A,B,C$ be symmetric sequences and let $X,Y,Z$ be sequences. There are natural isomorphisms
\begin{align*}
  (A\circ B)\circ C \Iso A\circ (B\circ C),\quad\quad
  (X\circhat Y)\circhat Z \Iso X\circhat (Y\circhat Z).
\end{align*}
\end{prop}

\begin{proof}
Using \eqref{eq:circ_calc} and \eqref{eq:tensor_power_sym}, there are natural isomorphisms
\begin{align*}
  A\circ(B\circ C)&\Iso\coprod\limits_{t\geq 0} A[\mathbf{t}]\tensor_{\Sigma_t}
  (B\circ C)^{\tensor t}
  \Iso \coprod\limits_{t\geq 0} A[\mathbf{t}]\tensor_{\Sigma_t}
  (B^{\tensor t})\circ C\\
  &\Iso\coprod\limits_{s\geq 0}\coprod\limits_{t\geq 0} 
  A[\mathbf{t}]\tensor_{\Sigma_t}(B^{\tensor t})[\mathbf{s}]\tensor_{\Sigma_s}
  C^{\tensor s}
  \Iso (A\circ B)\circ C,
\end{align*}
and a similar argument verifies the case for sequences.
\end{proof}

\subsection{Mapping sequences}

Let $B,C$ be symmetric sequences and $T\in\Sigma$. Let $Y,Z$ be sequences and $N\in\Omega$. There are functors
\begin{align*}
  \Sigma\times\Sigma^\op\rarrow\C,&\quad\quad
  (S,S')\longmapsto \Map(B[S,T],C[S']),\\
  \Omega\times\Omega^\op\rarrow\C,&\quad\quad
  (M,M')\longmapsto \Map(Y\langle M,N\rangle,Z[M']),
\end{align*}
which are useful for defining mapping objects of $(\SymSeq,\circ,I)$ and $(\Seq,\circhat,I)$.

\begin{defn}\label{def:map_seq} Let $B,C$ be symmetric sequences and $T\in\Sigma$. Let $Y,Z$ be sequences and $N\in\Omega$. The \emph{mapping sequences} $\Map^\circ(B,C)\in\SymSeq$ and $\Map^{\hat{\circ}}(Y,Z)\in\Seq$ are defined objectwise by the ends
\begin{align*}
  \Map^\circ(B,C)[T]\ &:=\ \Map((B^{\tensor T})[-],C)^\Sigma \ = \ 
  \Map(B[-,T],C)^\Sigma, \\
  \Map^{\hat{\circ}}(Y,Z)[N]\ &:=\ \Map((Y^{\tensorhat N})[-],Z)^\Omega 
  \ = \ \Map(Y\langle -,N\rangle,Z)^\Omega.
\end{align*}
\end{defn}

Hence $\Map^\circ(B,C)$ satisfies objectwise the universal property
\begin{align}\label{eq:mapping_sequence_universal}
\xymatrix{
   & & \Map(B[S,T],C[S])\ar[d]^-{([\id,\id],\zeta^*)} & S \\
   \cdot\ar@(u,l)[urr]^-{f[S]}
   \ar@(d,l)[drr]_-{f[S']}
   \ar@{.>}[r]^-{\bar{f}}_-{\existsunique}
   & \Map^\circ(B,C)[T]\ar@/^1pc/[ur]_-{\tau[S]}
  \ar@/_1pc/[dr]^-{\tau[S']}
  & \Map(B[S,T],C[S']) & \\
   & & \Map(B[S',T],C[S'])\ar[u]_-{([\zeta,\id],\id)} 
   & S'\ar[uu]_{\zeta}   
}
\end{align}
that each wedge $f$ factors uniquely through the terminal wedge $\tau$ of $\Map^\circ(B,C)[T]$. A similar universal property is satisfied objectwise by the mapping objects $\Map^{\hat{\circ}}(Y,Z)$. 

\begin{prop}\label{prop:calc_map_circ}
Let $B,C$ be symmetric sequences and $T\in\Sigma$, with $t:=|T|$. Let $Y,Z$ be sequences and $N\in\Omega$, with $n:=|N|$. There are natural isomorphisms,
\begin{align}
  \label{eq:mapping_calc}
  \Map^\circ(B,C)[\mathbf{t}]
  &\Iso \prod_{s\geq 0}\Map((B^{\tensor \mathbf{t}})[\mathbf{s}],
  C[\mathbf{s}])^{\Sigma_s}\Iso
  \prod_{s\geq 0}\Map(B[\mathbf{s},\mathbf{t}],
  C[\mathbf{s}])^{\Sigma_s},\\
  \notag
  \Map^{\hat{\circ}}(Y,Z)[\mathbf{n}]
  &\Iso
  \prod_{m\geq 0}\Map((Y^{\hat{\tensor} \mathbf{n}})[\mathbf{m}],
  Z[\mathbf{m}])\Iso
  \prod_{m\geq 0}\Map(Y\langle\mathbf{m},\mathbf{n}\rangle,Z[\mathbf{m}]).
\end{align}
\end{prop}

\begin{proof}
Using the universal property \eqref{eq:mapping_sequence_universal} for $\Map^\circ(B,C)[T]$ and restricting to a small skeleton $\Sigma'\subsetof\Sigma$, it is easy to obtain natural isormorphisms
\begin{align*}
  \Map^\circ(B,C)[\mathbf{t}]&\Iso\lim
  \biggl( 
  \xymatrix{
  \prod_{\substack{S\\ \text{in}\,\Sigma'}}\limits
  \Map(B[S,T],C[S])\ar@<2.5ex>[r]\ar@<1.5ex>[r] & 
  \prod_{\substack{\function{\zeta}{S'}{S}\\ \text{in}\,\Sigma'}}\limits 
  \Map(B[S,T],C[S'])
  }
  \biggr).
\end{align*}
This verifies \eqref{eq:mapping_calc} and the case for sequences is similar.
\end{proof}

These constructions give functors
\begin{align*}
  \SymSeq^\op\times\SymSeq\rarrow\SymSeq,&\quad\quad
  (B,C)\longmapsto\Map^\circ(B,C),\\
  \Seq^\op\times\Seq\rarrow\Seq,&\quad\quad
  (Y,Z)\longmapsto\Map^{\hat{\circ}}(Y,Z).
\end{align*}

\begin{prop}\label{prop:adjunction_circ}
Let $A,B,C$ be symmetric sequences and let $X,Y,Z$ be sequences. There are isomorphims 
\begin{align}\label{eq:adjunction_circ}
  \hom(A\circ B,C)&\Iso\hom(A,\Map^\circ(B,C)),\\
  \hom(X\circhat Y,Z)&\Iso\hom(X,\Map^{\hat{\circ}}(Y,Z)),
\end{align}
natural in $A,B,C$ and $X,Y,Z$.
\end{prop}

\begin{proof}
Using the universal properties \eqref{eq:circle_product_universal} together with \eqref{eq:mapping_sequence_universal} and the natural correspondence \eqref{eq:adjunction_C}, it is easy to verify that giving a map $A\circ B\rarrow C$ is the same as giving a map $A\rarrow\Map^\circ(B,C)$, and that the resulting correspondence is natural. A similar argument verifies the case for  sequences.
\end{proof}

\subsection{Monoidal structures}

\begin{prop}\label{prop:circ_monoidal}
$(\SymSeq,\circ,I)$ and $(\Seq,\hat{\circ},I)$ have the structure of closed monoidal categories with all small limits and colimits. Circle product is not symmetric. 
\end{prop}

\begin{proof}
Consider the case of symmetric sequences. To verify the monoidal structure, it is easy to use \eqref{eq:circ_calc} along with properties of $\tensor$ from Proposition~\ref{prop:tensor_monoidal} to describe the required natural isomorphisms and to verify the appropriate diagrams commute. Proposition~\ref{prop:adjunction_circ} verifies the monoidal structure is closed. Limits and colimits are calculated objectwise. A similar argument verifies the case for sequences.
\end{proof}

The following calculations follow directly from Proposition~\ref{prop:calc_map_circ}; similar calculations are true for sequences.

\begin{prop}\label{prop:map_seq_calc}
Let $B,C$ be symmetric sequences, $t\geq 0$, and $Z\in\C$. There are natural isomorphisms,
\begin{align}
  \notag
  \Map^\circ(B,*)&\Iso *,\\
  \notag
  \Map^\circ(\emptyset,C)[\mathbf{t}]&\Iso
  \left\{
  \begin{array}{rl}
  C[\mathbf{0}], & \text{for $t=0$,}\\
  *, & \text{otherwise,}
  \end{array}
  \right.\\
  \notag
  \Map^\circ(B,C)[\mathbf{0}]&\Iso C[\mathbf{0}],\\
  \notag
  \Map^\circ(B,C)[\mathbf{1}]&\Iso \prod_{s\geq 0}
  \Map(B[\mathbf{s}],C[\mathbf{s}])^{\Sigma_s},\\
  \label{eq:map_circle_with_hat_construction}
  \Map^\circ(\hat{Z},C)[\mathbf{t}]&\Iso\Map(Z^{\tensor t},C[\mathbf{0}]).
\end{align}
\end{prop}

\subsection{Circle products as Kan extensions}
Circle products can also be understood as Kan extensions. 

\begin{defn}
Let $\D,\E$ be categories and denote by $\xymatrix@1{\cdot\rightarrow\cdot}$ the category with exactly two objects and one non-identity morphism as indicated. The \emph{arrow category} $\Arrow(\D):=\D^{\cdot\rightarrow\cdot}$ is the category of $(\cdot\rightarrow\cdot)$-shaped diagrams in $\D$. Denote by $\ISO \E\subsetof\E$ the subcategory of isomorphisms of $\E$.
\end{defn}

Let $\F$ denote the category of finite sets and their maps and let $\OrdF$ denote the category of totally ordered finite sets and their order preserving maps. For each $A,B\in\SymSeq$ and $X,Y\in\Seq$, there are functors defined objectwise by
\begin{align*}
  \function{A*B}{\bigl(\ISO\Ar(\F)\bigr)^\op}{\C},&\quad\quad
  (S\xrightarrow{\pi}T)\longmapsto A[T]\tensor 
  \tensor_{t\in T}B[\pi^{-1}(t)], \\
  \function{X*Y}{\bigl(\ISO\Ar(\OrdF)\bigr)^\op}{\C},&\quad\quad
  (M\xrightarrow{\pi} N)\longmapsto X[N]\tensor 
  \tensor_{n\in N}Y[\pi^{-1}(n)].
\end{align*}
It is easy to check that the circle products are left Kan extensions of $-*-$ along projection onto source,
\begin{align*}
\xymatrix{
  \bigl(\ISO\Ar(\F)\bigr)^{\op}
  \ar[d]^{\mathsf{proj}}\ar[rr]^-{A*B} & & \C \\
  \Sigma^{\op}\ar[rr]^{A\circ B}_{\text{left Kan extension}} & & \C,
} &&
\xymatrix{
  \bigl(\ISO\Ar(\OrdF)\bigr)^{\op}
  \ar[d]^{\mathsf{proj}}\ar[rr]^-{X*Y} & & \C \\
  \Omega^{\op}\ar[rr]^{X\circhat Y}_{\text{left Kan extension}} & & \C.
}
\end{align*}

\section{Algebras and modules over operads}
\label{sec:operads_modules_algebras}

The purpose of this section is to recall certain definitions and constructions involving (symmetric) sequences and algebras and modules over operads, including certain constructions of colimits and limits which will play a fundamental role in several of the main arguments in this paper. A useful introduction to operads and their algebras is given in \cite{Kriz_May}; see also the original article \cite{May}. Other accounts include \cite{Fresse, Fresse_modules, Hinich, Markl_Shnider_Stasheff, McClure_Smith_intro, Rezk, Spitzweck}. The material in this section is largely influenced by \cite{Rezk}.

\begin{defn}\label{defn:operads}\quad
\begin{itemize}
\item A \emph{$\Sigma$-operad} is a monoid object in $(\SymSeq,\circ,I)$ and a \emph{morphism of $\Sigma$-operads} is a morphism of monoid objects in $(\SymSeq,\circ,I)$.
\item An \emph{$\Omega$-operad} is a monoid object in $(\Seq,\hat{\circ},I)$ and a \emph{morphism of $\Omega$-operads} is a morphism of monoid objects in $(\Seq,\hat{\circ},I)$.
\end{itemize}
\end{defn}

These two types of operads were originally defined in \cite{May}; the $\Sigma$-operad has symmetric groups and the $\Omega$-operad is based on ordered sets and is called a non-$\Sigma$ operad \cite{Kriz_May, May}. For a useful introduction to monoid objects, see \cite[VII]{MacLane_categories}.

\begin{example}
More explicitly, for instance, a $\Sigma$-operad is a symmetric sequence $\capO$ together with maps $\function{m}{\capO\circ\capO}{\capO}$ and $\function{\eta}{I}{\capO}$ in $\SymSeq$ which make the diagrams
\begin{align*}
  \xymatrix{
  \capO\circ\capO\circ\capO\ar[r]^-{m\circ\id}\ar[d]^{\id\circ m} & 
  \capO\circ\capO\ar[d]^{m}\\
  \capO\circ\capO\ar[r]^-{m} & \capO
} \quad\quad
\xymatrix{
  I\circ\capO\ar[r]^{\eta\circ\id}\ar[d]^{\iso} & 
  \capO\circ\capO\ar[d]^{m} & \capO\circ I\ar[d]_{\iso}\ar[l]_{\id\circ\eta}\\
  \capO\ar@{=}[r] & \capO & 
  \capO\ar@{=}[l]
}
\end{align*}
commute. If $\capO$ and $\capO'$ are $\Sigma$-operads, then a morphism of $\Sigma$-operads is a map $\function{f}{\capO}{\capO'}$ in $\SymSeq$ which makes the diagrams
\begin{align*}
  \xymatrix{
  \capO\circ\capO\ar[r]^-{m}\ar[d]^{f\circ f} & \capO\ar[d]^{f}\\
  \capO'\circ\capO'\ar[r]^-{m} & \capO'
} \quad\quad
\xymatrix{
  \capO\ar[d]^{f} & I\ar[l]_-{\eta}\ar@{=}[d]\\
  \capO' & I\ar[l]_-{\eta}
  }
\end{align*}
commute.
\end{example}

\subsection{Algebras and modules over operads}
\label{sec:algebras_over_operads}
Similar to the case of any monoid object, we introduce operads because we are interested in the objects they act on. The reader may wish to compare the following definition with \cite[Chapter VI]{MacLane_homology}.

\begin{defn}
Let $Q,R,S$ be $\Sigma$-operads (resp. $\Omega$-operads).
\begin{itemize}
\item A \emph{left $R$-module} is an object in $(\SymSeq,\circ,I)$ (resp. an object in $(\Seq,\hat{\circ},I)$) with a left action of $R$ and a \emph{morphism of left $R$-modules} is a map which respects the left $R$-module structure.
\item A \emph{right $S$-module} is an object in $(\SymSeq,\circ,I)$ (resp. an object in $(\Seq,\hat{\circ},I)$) with a right action of $S$ and a \emph{morphism of right $S$-modules} is a map which respects the right $S$-module structure.
\item An \emph{$(R,S)$-bimodule} is an object in $(\SymSeq,\circ,I)$ (resp. an object in $(\Seq,\hat{\circ},I)$) with compatible left $R$-module and right $S$-module structures and a \emph{morphism of $(R,S)$-bimodules} is a map which respects the $(R,S)$-bimodule structure.
\end{itemize}
\end{defn}

Each $\Sigma$-operad $\capO$ (resp. $\Omega$-operad $\capO$) determines a functor $\function{\capO}{\C}{\C}$ defined objectwise by
\begin{align*}
  \capO(X)&:=\capO\circ(X)
  =\coprod\limits_{t\geq 0}\capO[\mathbf{t}]
  \tensor_{\Sigma_t}X^{\tensor t},\\
  \Bigl(\text{resp.}\quad
  \capO(X)&:=\capO\circhat(X)
  =\coprod\limits_{n\geq 0}\capO[\mathbf{n}]
  \tensor X^{\tensor n}
  \Bigr),
\end{align*}
along with natural transformations $\function{m}{\capO\capO}{\capO}$ and $\function{\eta}{\id}{\capO}$ which give the functor $\function{\capO}{\C}{\C}$ the structure of a monad (or triple) in $\C$. For a useful introduction to monads and their algebras, see \cite[VI]{MacLane_categories}. 

Recall the following definition from \cite[I.2 and I.3]{Kriz_May}. 

\begin{defn}
Let $\capO$ be an operad. An \emph{$\capO$-algebra} is an object in $\C$ with a left action of the monad $\functor{\capO}{\C}{\C}$ and a \emph{morphism of $\capO$-algebras} is a map in $\C$ which respects the left action of the monad $\function{\capO}{\C}{\C}$.
\end{defn}

One perspective offered in \cite[I.3]{Kriz_May} is that operads determine particularly manageable monads. From this perspective, operads correspond to special functors in such a way that circle product corresponds to composition, but because these functors have such simple descriptions in terms of (symmetric) sequences, operads are easier to work with than arbitrary functors.

It is easy to verify that an $\capO$-algebra is the same as an object $X\in\C$ with a left $\capO$-module structure on $\hat{X}$, and if $X$ and $X'$ are $\capO$-algebras, then a morphism of $\capO$-algebras is the same as a map $\function{f}{X}{X'}$ in $\C$ such that $\function{\hat{f}}{\hat{X}}{\hat{X'}}$ is a morphism of left $\capO$-modules. In other words, an algebra over an operad $\capO$ is the same as a left $\capO$-module which is concentrated at $0$. Giving a symmetric sequence $Y$ a left $\capO$-module structure is the same as giving a morphism of operads 
\begin{align}\label{eq:operad_action}
  \function{m}{\capO}{\Map^\circ(Y,Y)}.
\end{align}
Similarly, giving an object $X\in\C$ an $\capO$-algebra structure is the same as giving a morphism of operads
\begin{align*}
  \function{m}{\capO}{\Map^\circ(\hat{X},\hat{X})}.
\end{align*}
This is the original definition given in \cite{May} of an $\capO$-algebra structure on $X$, where $\Map^\circ(\hat{X},\hat{X})$ is called the \emph{endomorphism operad} of $X$, and motivates the suggestion in \cite{Kriz_May, May} that $\capO[\mathbf{t}]$ should be thought of as parameter objects for $t$-ary operations. 

\subsection{Reflexive coequalizers and filtered colimits}
Reflexive coequalizers will be useful for building colimits in the categories of algebras and modules over an operad.

\begin{defn}
\label{defn:reflexive_pair}
A pair of maps of the form $\xymatrix@1{X_0 & X_1\ar@<-0.5ex>[l]_-{d_0}\ar@<+0.5ex>[l]^-{d_1}}$ in $\C$ is called a \emph{reflexive pair} if there exists $\function{s_0}{X_0}{X_1}$ in $\C$ such that $d_0s_0=\id$ and $d_1s_0=\id$. A \emph{reflexive coequalizer} is the coequalizer of a reflexive pair. 
\end{defn}

The following proposition is proved in \cite[Lemma 2.3.2]{Rezk}. Part (a) also follows from the proof of \cite[Proposition II.7.2]{EKMM} or the arguments in \cite[Section 1]{Goerss_Hopkins}.

\begin{prop}\label{prop:colimit_properties_tensor}
\quad
\begin{itemize}
\item[(a)] If  
$\xymatrix@1{X_{-1} & X_0\ar[l] & X_1\ar@<-0.5ex>[l]\ar@<+0.5ex>[l]}$ and 
$\xymatrix@1{Y_{-1} & Y_0\ar[l] & Y_1\ar@<-0.5ex>[l]\ar@<+0.5ex>[l]}$
are reflexive coequalizer diagrams in $\C$, then their objectwise tensor product
\begin{align}
\label{eq:objectwise_tensor_reflexive_coequalizer}
\xymatrix{
  X_{-1}\tensor Y_{-1} & X_0\tensor Y_0\ar[l]& 
  X_1\tensor Y_1\ar@<-0.5ex>[l]\ar@<+0.5ex>[l]
}
\end{align}
is a reflexive coequalizer diagram in $\C$.
\item[(b)] If $\functor{X,Y}{\D}{\C}$ are filtered diagrams, then objectwise tensor product of their colimiting cones is a colimiting cone. In particular, there are natural isomorphisms
\begin{align*}
  \colim_{d\in\D}(X_d\tensor Y_d)\Iso
  (\colim_{d\in\D}X_d)\tensor(\colim_{d\in \D}Y_d).
\end{align*}
in $\C$.
\end{itemize}
\end{prop}

\begin{proof}
Consider part (a). We want to verify that \eqref{eq:objectwise_tensor_reflexive_coequalizer} is a coequalizer diagram; it is sufficient to verify the universal property of colimits. Using the diagram
\begin{align*}
\xymatrix{
  X_0\tensor Y_1\ar@<-0.5ex>[d]\ar@<0.5ex>[d]
   & X_1\tensor Y_1\ar@<-0.5ex>[l]\ar@<0.5ex>[l]
  \ar@<-0.5ex>[d]\ar@<0.5ex>[d]   \\
  X_0\tensor Y_0 & X_1\tensor Y_0\ar@<-0.5ex>[l]\ar@<0.5ex>[l]
}
\end{align*}
and the map $X_0\tensor Y_0\rightarrow X_{-1}\tensor Y_{-1}$, together with the maps $s_0$ in Definition \ref{defn:reflexive_pair} and the relations satisfied by the reflexive pairs, it is easy to verify that \eqref{eq:objectwise_tensor_reflexive_coequalizer} satisfies the universal property of a coequalizer diagram; note that tensoring with any $X\in\C$ preserves colimiting cones. Further details are given in the proof of \cite[Proposition II.7.2]{EKMM} and the argument appearing between Definition 1.8 and Lemma 1.9 in \cite[Section 1]{Goerss_Hopkins}. Verification of (b) is similar to (a), except we use the properties satisfied by filtered diagrams instead of reflexive pairs.
\end{proof}

Hence objectwise tensor product of diagrams in $(\C,\tensor,\unit)$ respects certain colimiting cones. Objectwise circle product of diagrams in $(\SymSeq,\circ,I)$ and $(\Seq,\hat{\circ},I)$ behave similarly. The following proposition is proved in \cite[Lemma 2.3.4]{Rezk}. Part (a) also follows from the proof of \cite[Proposition II.7.2]{EKMM} or the arguments in \cite[Section 1]{Goerss_Hopkins}.

\begin{prop}\label{prop:colimit_properties_circle}
\quad
\begin{itemize}
\item[(a)] Suppose 
$\xymatrix@1{A_{-1} & A_0\ar[l] & A_1\ar@<-0.5ex>[l]\ar@<+0.5ex>[l]}$ and 
$\xymatrix@1{B_{-1} & B_0\ar[l] & B_1\ar@<-0.5ex>[l]\ar@<+0.5ex>[l]}$
are reflexive coequalizer diagrams in $\SymSeq$. Then their objectwise circle product
\begin{align}
\label{eq:reflexive_coequalizer}
\xymatrix{
  A_{-1}\circ B_{-1} & A_0\circ B_0\ar[l]& 
  A_1\circ B_1\ar@<-0.5ex>[l]\ar@<+0.5ex>[l]
}
\end{align}
is a reflexive coequalizer diagram in $\SymSeq$.
\item[(b)] Suppose $\functor{A,B}{\D}{\SymSeq}$ are filtered diagrams. Then objectwise circle product of their colimiting cones is a colimiting cone. In particular, there  are natural isomorphisms
\begin{align*}
  \colim_{d\in\D}(A_d\circ B_d)\Iso
  (\colim_{d\in\D}A_d)\circ(\colim_{d\in \D}B_d).
\end{align*}
in $\SymSeq$.
\item[(c)] For sequences, the corresponding statements in (a) and (b) remain true; i.e., when $(\SymSeq,\circ,I)$ is replaced by $(\Seq,\hat{\circ},I)$.
\end{itemize}
\end{prop}

\begin{proof}
This is because circle products are constructed universally from iterated tensor powers. In other words, this follows easily from Proposition \ref{prop:colimit_properties_tensor} together with \eqref{eq:tensor_powers} and \eqref{eq:circle_product_universal}, by verifying the universal property of  colimits. 
\end{proof}

The following is a special case of particular interest.

\begin{prop}\quad
\begin{itemize}
\item[(a)] Suppose 
$\xymatrix@1{A_{-1} & A_0\ar[l] & A_1\ar@<-0.5ex>[l]\ar@<+0.5ex>[l]}$ is a reflexive coequalizer diagram in $\SymSeq$ and 
$\xymatrix@1{Z_{-1} & Z_0\ar[l] & Z_1\ar@<-0.5ex>[l]\ar@<+0.5ex>[l]}$
is a reflexive coequalizer diagram in $\C$. Then their objectwise evaluation 
\begin{align*}
\xymatrix{
  A_{-1}\circ(Z_{-1}) & A_0\circ(Z_0)\ar[l]& 
  A_1\circ(Z_1)\ar@<-0.5ex>[l]\ar@<+0.5ex>[l]
}
\end{align*}
is a reflexive coequalizer diagram in $\C$.
\item[(b)] Suppose $\functor{A}{\D}{\SymSeq}$ and $\functor{Z}{\D}{\C}$ are filtered diagrams. Then objectwise evaluation of their colimiting cones is a colimiting cone. In particular, there  are natural isomorphisms
\begin{align*}
  \colim_{d\in\D}\bigl(A_d\circ(Z_d)\bigr)\Iso
  (\colim_{d\in\D}A_d)\circ(\colim_{d\in \D}Z_d).
\end{align*}
in $\C$.
\item[(c)] For sequences, the corresponding statements in (a) and (b) remain true; i.e., when $(\SymSeq,\circ,I)$ is replaced by $(\Seq,\hat{\circ},I)$.
\end{itemize}
\end{prop}

\begin{proof}
This follows from Proposition~\ref{prop:colimit_properties_circle} using the embedding \eqref{eq:embedding_at_zero}.
\end{proof}

\subsection{Free-forgetful adjunctions}

\begin{defn} Let $\capO,R,S$ be operads.
\begin{itemize}
\item $\LtO$ is the category of left $\capO$-modules and their morphisms.
\item $\RtO$ is the category of right $\capO$-modules and their morphisms.
\item $\Bi_{(R,S)}$ is the category of $(R,S)$-bimodules and their morphisms.
\item $\AlgO$ is the category of $\capO$-algebras and their morphisms.
\end{itemize}
\end{defn}

The following free-forgetful adjunctions will be useful.

\begin{prop}\quad\label{prop:free_forgetful_adjunctions}
\begin{itemize}
\item[(a)] Let $\capO,R,S$ be $\Sigma$-operads. There are adjunctions
\begin{align*}
\xymatrix{
  \SymSeq\ar@<0.5ex>[r]^-{\capO\circ-} & \LtO,\ar@<0.5ex>[l]^-{U}
}\quad\quad
\xymatrix{
  \SymSeq\ar@<0.5ex>[r]^-{-\circ\capO} & \RtO,\ar@<0.5ex>[l]^-{U}
}\quad\quad
\xymatrix{
  \SymSeq\ar@<0.5ex>[r]^-{R\circ-\circ S} & 
  \Bi_{(R,S)},\ar@<0.5ex>[l]^-{U}
}
\end{align*}
with left adjoints on top and $U$ the forgetful functor.
\item[(b)] Let $\capO,R,S$ be $\Omega$-operads. There are adjunctions
\begin{align*}
\xymatrix{
  \Seq\ar@<0.5ex>[r]^-{\capO\circhat-} & \LtO,\ar@<0.5ex>[l]^-{U}
}\quad\quad
\xymatrix{
  \Seq\ar@<0.5ex>[r]^-{-\circhat\capO} & \RtO,\ar@<0.5ex>[l]^-{U}
}\quad\quad
\xymatrix{
  \Seq\ar@<0.5ex>[r]^-{R\circhat-\circhat S} & 
  \Bi_{(R,S)},\ar@<0.5ex>[l]^-{U}
}
\end{align*}
with left adjoints on top and $U$ the forgetful functor.
\end{itemize}
\end{prop}

The following are the corresponding free-forgetful adjunctions for algebras over an operad.
\begin{prop}\label{prop:free_forgetful_algebras}
Let $\capO$ be a $\Sigma$-operad and $\capO'$ be an $\Omega$-operad. There are adjunctions
\begin{align*}
\xymatrix{
  \C\ar@<0.5ex>[r]^-{\capO\circ(-)} & \AlgO,\ar@<0.5ex>[l]^-{U}
}\quad\quad
\xymatrix{
  \C\ar@<0.5ex>[r]^-{\capO'\circhat(-)} & \Alg_{\capO'},\ar@<0.5ex>[l]^-{U}
}
\end{align*}
with left adjoints on top and $U$ the forgetful functor.
\end{prop}

\subsection{Construction of colimits and limits}

The following proposition is proved in \cite[Proposition 2.3.5]{Rezk}, and is closely related to \cite[Proposition II.7.2]{EKMM}. Since it plays a fundamental role in several of the main arguments in this paper, we have included a proof below.

\begin{prop}\label{prop:reflexive_and_filtered_colimits}
Let $\capO,R,S$ be operads. Reflexive coequalizers and filtered colimits exist in $\LtO$, $\RtO$, $\Bi_{(R,S)}$, and $\AlgO$, and are preserved by the forgetful functors.
\end{prop}

\begin{proof}
Let $\capO$ be a $\Sigma$-operad and consider the case of left $\capO$-modules. Suppose $\xymatrix@1{A_0 & A_1\ar@<-0.5ex>[l]\ar@<+0.5ex>[l]}$ is a reflexive pair in $\LtO$ and consider the solid commutative diagram
\begin{align*}
\xymatrix{
  \capO\circ\capO\circ A_{-1}\ar@{.>}@<-0.5ex>[d]_{d_0}
  \ar@{.>}@<0.5ex>[d]^{d_1} & 
  \capO\circ\capO\circ A_0\ar[l]\ar@<-0.5ex>[d]_{m\circ\id}
  \ar@<0.5ex>[d]^{\id\circ m} & 
  \capO\circ\capO\circ A_1\ar@<-0.5ex>[l]\ar@<0.5ex>[l]
  \ar@<-0.5ex>[d]_{m\circ\id}\ar@<0.5ex>[d]^{\id\circ m} \\
  \capO\circ A_{-1}\ar@{.>}[d]^{m} & \capO\circ A_0\ar[l]\ar[d]^{m} & 
  \capO\circ A_1\ar@<-0.5ex>[l]\ar@<0.5ex>[l]\ar[d]^{m} \\
  A_{-1}\ar@{.>}@/^1pc/[u]^{s_0} & A_0\ar[l]\ar@/^1pc/[u]^{\eta\circ\id} & 
  A_1\ar@<-0.5ex>[l]\ar@<0.5ex>[l]
  \ar@/^1pc/[u]^{\eta\circ\id}
}
\end{align*}
in $\SymSeq$, with bottom row the reflexive coequalizer diagram of the underlying reflexive pair in $\SymSeq$. By Proposition~\ref{prop:colimit_properties_circle}, the rows are reflexive coequalizer diagrams and hence there exist unique dotted arrows $m,s_0,d_0,d_1$ in $\SymSeq$ which make the diagram commute. By uniqueness, $s_0=\eta\circ\id$, $d_0=m\circ\id$, and $d_1=\id\circ m$. It is easy to verify that $m$ gives $A_{-1}$ the structure of a left $\capO$-module and that the bottom row is a reflexive coequalizer diagram in $\LtO$; it is easy to check the diagram lives in $\LtO$ and that the colimiting cone is initial with respect to all cones in $\LtO$. The case for filtered colimits is similar. The cases for $\RtO$, $\Bi_{(R,S)}$, and $\AlgO$ can be argued similarly. A similar argument verifies the case for $\Omega$-operads.
\end{proof}

The following proposition is proved in \cite[Proposition 2.3.5]{Rezk}. It also follows from the argument in \cite[Proposition II.7.4]{EKMM}. 

\begin{prop}\label{prop:colimits_exist}
Let $\capO,R,S$ be operads. All small colimits exist in $\LtO$, $\RtO$, $\Bi_{(R,S)}$, and $\AlgO$. If $\functor{A}{\D}{\LtO}$ is a small diagram, then $\colim A$ in $\LtO$ may be calculated by a reflexive coequalizer of the form
\begin{align*}
\colim A\Iso\colim\Bigl(
\xymatrix{
\capO\circ\bigl(\colim\limits_{d\in D} A_d\bigr) & 
  \capO\circ\bigl(\colim\limits_{d\in D}(\capO\circ A_d)\bigr)
  \ar@<-1.0ex>[l]\ar@<0.0ex>[l]
}
\Bigr)
\end{align*}
in the underlying category $\SymSeq$; the colimits appearing inside the parenthesis are in the underlying category $\SymSeq$.
\end{prop}

\begin{example}
For instance, if $\capO$ is a $\Sigma$-operad and $A,B\in\LtO$, then the coproduct $A\amalg B$ in $\LtO$ may be calculated by a reflexive coequalizer of the form
\begin{align*}
 A\amalg B\Iso\colim\Bigl(
 \xymatrix{
 \capO\circ(A\amalg B) & 
 \capO\circ\bigl((\capO\circ A)\amalg 
 (\capO\circ B)\bigr)\ar@<-1.0ex>[l]\ar@<0.0ex>[l]
 }\Bigr)
\end{align*}
in the underlying category $\SymSeq$. The coproducts appearing inside the parentheses are in the underlying category $\SymSeq$.
\end{example}

Colimits in right modules over an operad are particularly simple to calculate.

\begin{prop}
Let $\capO$ be an operad. The forgetful functors from right $\capO$-modules in Proposition~\ref{prop:free_forgetful_adjunctions} preserve all small colimits.
\end{prop}

\begin{proof}
This is because the functor $\functor{-\circ\capO}{\SymSeq}{\SymSeq}$ is a left adjoint. A similar argument verifies the case for $\Omega$-operads.
\end{proof}

Limits in algebras and modules over an operad are also simple to calculate.

\begin{prop}
\label{prop:limits_exist}
Let $\capO,R,S$ be operads. All small limits exist in $\LtO$, $\RtO$, $\Bi_{(R,S)}$, and $\AlgO$, and are preserved by the forgetful functors in Propositions \ref{prop:free_forgetful_adjunctions} and \ref{prop:free_forgetful_algebras}. 
\end{prop}

\subsection{Circle products (mapping sequences) over an operad}
In this section we present some basic constructions for modules. The reader may wish to compare the following definition with \cite[Chapter VI.5]{MacLane_homology}.

\begin{defn}
Let $R$ be a $\Sigma$-operad (resp. $\Omega$-operad), $A$ a right $R$-module, and $B$ a left $R$-module. Define $A\circ_R B\in\SymSeq$ (resp. $A\,\hat{\circ}_R\, B\in\Seq$) by the reflexive coequalizer
\begin{align*}
  A\circ_R B &:=\colim
  \Bigl(
  \xymatrix{
  A\circ B & A\circ R\circ B
  \ar@<0.5ex>[l]^-{d_1}\ar@<-0.5ex>[l]_-{d_0}
  }
  \Bigr),\\
  \Bigl(\text{resp.}\quad
  A\,\hat{\circ}_R\, B &:=\colim
  \Bigl(
  \xymatrix{
  A\circhat B & A\circhat R\circhat B
  \ar@<0.5ex>[l]^-{d_1}\ar@<-0.5ex>[l]_-{d_0}
  }
  \Bigr)
  \Bigr),
\end{align*}
with $d_0$ induced by $\function{m}{A\circ R}{A}$ and $d_1$ induced by $\function{m}{R\circ B}{B}$ (resp. $d_0$ induced by $\function{m}{A\circhat R}{A}$ and $d_1$ induced by $\function{m}{R\circhat B}{B}$).
\end{defn}

\begin{defn}
Let $S$ be a $\Sigma$-operad (resp. $\Omega$-operad) and let $B$ and $C$ be right $S$-modules. Define $\Map^\circ_S(B,C)\in\SymSeq$ (resp. $\Map^{\hat{\circ}}_S(B,C)\in\Seq$) by the equalizer
\begin{align*}
  \Map^\circ_S(B,C)&:=\lim
  \Bigl( 
  \xymatrix{
  \Map^\circ(B,C)\ar@<-0.5ex>[r]_-{d^1}\ar@<0.5ex>[r]^-{d^0} & 
  \Map^\circ(B\circ S,C)
  }
  \Bigr),\\
  \Bigl(\text{resp.}\quad
  \Map^{\hat{\circ}}_S(B,C)&:=\lim
  \Bigl( 
  \xymatrix{
  \Map^{\hat{\circ}}(B,C)\ar@<-0.5ex>[r]_-{d^1}\ar@<0.5ex>[r]^-{d^0} & 
  \Map^{\hat{\circ}}(B\circhat S,C)
  }
  \Bigr)
  \Bigr),
\end{align*}
with $d^0$ induced by $\function{m}{B\circ S}{B}$ and $d^1$ induced by $\function{m}{C\circ S}{C}$ (resp. $d^0$ induced by $\function{m}{B\circhat S}{B}$ and $d^1$ induced by $\function{m}{C\circhat S}{C}$).
\end{defn}

\subsection{Adjunctions}
The reader may wish to compare the following adjunctions with \cite[Chapter VI.8]{MacLane_homology}.
\begin{prop}\label{prop:relative_adjunctions}
Let $Q,R,S$ be $\Sigma$-operads. There are isomorphisms,
\begin{align}
  \hom_{\Rt_S}(A\circ B,C)&\Iso\hom(A,\Map^\circ_S(B,C)),
  \label{eq:adj_first}\\
  \hom(A\circ_R B,C)&\Iso\hom_{\Rt_R}(A,\Map^\circ(B,C)),
  \label{eq:adj_second}\\
  \hom_{(Q,S)}(A\circ_R B,C)&\Iso\hom_{(Q,R)}(A,\Map^\circ_S(B,C)),
  \label{eq:adj_third}
  \end{align}
natural in $A,B,C$. 
\end{prop}

\begin{rem}
In \eqref{eq:adj_first}, $A$ is a symmetric sequence, and both $B$ and $C$ have right $S$-module structures. In \eqref{eq:adj_second}, $A$ has a right $R$-module structure, $B$ has a left $R$-module structure, and $C$ is a symmetric sequence. In \eqref{eq:adj_third}, $A$ has a $(Q,R)$-bimodule structure, $B$ has a $(R,S)$-bimodule structure, and $C$ has a $(Q,S)$-bimodule structure. There is a corresponding statement for $\Omega$-operads.
\end{rem}

\subsection{Change of operads}

\begin{prop}\label{prop:change_of_operads}
Let $\function{f}{R}{S}$ be a morphism of $\Sigma$-operads (resp. $\Omega$-operads). There are adjunctions
\begin{align*}
\xymatrix{
  \Lt_R\ar@<0.5ex>[r]^-{f_*} & \Lt_S,\ar@<0.5ex>[l]^-{f^*}
}\quad\quad
\xymatrix{
  \Alg_R\ar@<0.5ex>[r]^-{f_*} & \Alg_S,\ar@<0.5ex>[l]^-{f^*}
} 
\end{align*}
with left adjoints on top and $f^*$ the forgetful functor. Here, the left-hand adjunction satisfies $f_*:=S\circ_R-$ (resp. $f_*:=S\,\hat{\circ}_R\,-$).
\end{prop}

\section{Homotopical analysis of circle products and tensor products}

The purpose of this section is to prove Theorems \ref{thm:well_behaved_tensor} and \ref{thm:circ_compatible}, which verify that the projective model structure on (symmetric) sequences meshes nicely with tensor products and circle products. So far in this paper, except in Section~\ref{sec:introduction}, we have only used the property that $(\C,\tensor,\unit)$ is a closed symmetric monoidal category with all small limits and colimits. In this section, we begin to make use of the model category assumptions on $(\C,\tensor,\unit)$ described in Basic Assumption~\ref{BasicAssumption}.

It is easy to check that the diagram categories $\SymSeq$ and $\Seq$ inherit corresponding projective model category structures, where the weak equivalences (resp. fibrations) are the objectwise weak equivalences (resp. objectwise fibrations). Each of these model structures is cofibrantly generated in which the generating cofibrations and acyclic cofibrations have small domains. The following theorem, which is proved in Section~\ref{sec:pushout_corner_proof}, verifies that these model structures give the closed symmetric monoidal categories $(\SymSeq,\tensor,1)$ and $(\Seq,\hat{\tensor},1)$ the structure of monoidal model categories.

\begin{thm}\quad\label{thm:well_behaved_tensor}
\begin{itemize}
\item[(a)] In symmetric sequences (resp. sequences), if $\function{i}{K}{L}$ and $\function{j}{A}{B}$ are cofibrations, then the pushout corner map
\begin{align*}
  &\xymatrix{
  L\tensor A\coprod_{K\tensor A} K\tensor B\ar[r] & L\tensor B
  },\\
  \bigl(\text{resp.}\quad
  &\xymatrix{
  L\hat{\tensor} A\coprod_{K\hat{\tensor} A} K\hat{\tensor} B\ar[r] & 
  L\hat{\tensor} B
  }\bigr),
\end{align*}
is a cofibration that is an acyclic cofibration if either $i$ or $j$ is a weak equivalence.
\item[(b)] In symmetric sequences (resp. sequences), if $\function{j}{A}{B}$ is a cofibration and $\function{p}{X}{Y}$ is a fibration, then the pullback corner map
\begin{align*}
  &\xymatrix{
  \Map^\tensor(B,X)\ar[r] & \Map^\tensor(A,X)\times_{\Map^\tensor(A,Y)}
  \Map^\tensor(B,Y)
  },\\
  \bigl(\text{resp.}\quad
  &\xymatrix{
  \Map^{\hat{\tensor}}(B,X)\ar[r] & \Map^{\hat{\tensor}}(A,X)
  \times_{\Map^{\hat{\tensor}}(A,Y)}
  \Map^{\hat{\tensor}}(B,Y)
  }\bigr),
\end{align*}
is a fibration that is an acyclic fibration if either $j$ or $p$ is a weak equivalence.
\end{itemize}
\end{thm}

These model structures also mesh nicely with circle product, provided an additional cofibrancy condition is satisfied.  The following theorem, which is proved in Section~\ref{sec:pushout_corner_proof}, is motivated by a similar argument given in \cite{Rezk} for symmetric sequences of simplicial sets using a model structure with fewer weak equivalences.

\begin{thm}\label{thm:circ_compatible}
Let $A$ be a cofibrant symmetric sequence (resp. cofibrant sequence). 
\begin{itemize}
\item[(a)] In symmetric sequences (resp. sequences), if $\function{i}{K}{L}$ and $\function{j}{A}{B}$ are cofibrations, then the pushout corner map 
\begin{align*}
  &\xymatrix{
    L\circ A\coprod_{K\circ A} K\circ B\ar[r] & L\circ B},\\
  \bigl(\text{resp.}\quad
  &\xymatrix{
    L\circhat A\coprod_{K\circhat A} K\circhat B\ar[r] & L\circhat B
  }\bigr),
\end{align*}
is a cofibration that is an acyclic cofibration if either $i$ or $j$ is a weak equivalence.
\item[(b)] In symmetric sequences (resp. sequences), if $\function{j}{A}{B}$ is a cofibration and $\function{p}{X}{Y}$ is a fibration, then the pullback corner map
\begin{align*}
  &\xymatrix{
  \Map^\circ(B,X)\ar[r] & \Map^\circ(A,X)\times_{\Map^\circ(A,Y)}
  \Map^\circ(B,Y)
  },\\
  \bigl(\text{resp.}\quad
  &\xymatrix{
  \Map^{\hat{\circ}}(B,X)\ar[r] & \Map^{\hat{\circ}}(A,X)
  \times_{\Map^{\hat{\circ}}(A,Y)}
  \Map^{\hat{\circ}}(B,Y)
  }\bigr),
\end{align*}
is a fibration that is an acyclic fibration if either $j$ or $p$ is a weak equivalence.
\end{itemize}
\end{thm}

Consider Theorem~\ref{thm:circ_compatible} and assume that $A=\emptyset$. It is useful to note that $L\circ\emptyset$ and $K\circ\emptyset$ may not be isomorphic, and similarly $\Map^\circ(\emptyset,X)$ and $\Map^\circ(\emptyset,Y)$ may not be isomorphic. On the other hand, Theorem~\ref{thm:circ_compatible} reduces the proof of the following proposition to a trivial inspection at the emptyset $\mathbf{0}$. 

\begin{prop}
Let $B$ be a cofibrant symmetric sequence (resp. cofibrant sequence).
\begin{itemize}
\item[(a)] In symmetric sequences (resp. sequences), if $\function{i}{K}{L}$ is a cofibration, then the induced map
\begin{align*}
  & K\circ B\rarrow L\circ B\ ,\quad
  \bigl(\text{resp.}\quad
  K\circhat B\rarrow L\circhat B\
  \bigr),
\end{align*}
is a cofibration that is an acyclic cofibration if $i$ is a weak equivalence.
\item [(b)] In symmetric sequences (resp. sequences), if $\function{p}{X}{Y}$ is a fibration, then the induced map 
\begin{align*}
  &\Map^\circ(B,X)\rarrow\Map^\circ(B,Y)\ , \quad
  \bigl(\text{resp.}\quad \Map^{\hat{\circ}}(B,X)\rarrow\Map^{\hat{\circ}}(B,Y)\
  \bigr),
\end{align*}
is a fibration that is an acyclic fibration if p is a weak equivalence.
\end{itemize}
\end{prop}

\begin{proof}
Statements (a) and (b) are equivalent, hence it is sufficient to verify (b).
Suppose $B$ is cofibrant and $\function{p}{X}{Y}$ is an acyclic fibration. We want to verify each induced map
\begin{align*}
  \Map^\circ(B,X)[\mathbf{t}]\rarrow\Map^\circ(B,Y)[\mathbf{t}]
\end{align*}
is an acyclic fibration in $\C$. Theorem~\ref{thm:circ_compatible}(b) implies this for $t\geq 1$. For $t=0$, it is enough to note that $X[\mathbf{0}]\rarrow Y[\mathbf{0}]$ is an acyclic fibration. The other case is similar.
\end{proof}

\subsection{Fixed points, pullback corner maps, and tensor products}

If $G$ is a finite group, it is easy to check that the diagram category $\C^{G^\op}$ inherits a corresponding projective model category structure, where the weak equivalences (resp. fibrations) are the objectwise weak equivalences (resp. objectwise fibrations).

\begin{prop}\label{prop:corner_map_fixed}
Let $G$ be a finite group and $H\subset G$ a subgroup. In $\C^{G^\op}$, suppose $\function{j}{A}{B}$ is a cofibration and $\function{p}{X}{Y}$ is a fibration. Then in $\C$ the pullback corner map
\begin{align*}
  \xymatrix{
  \Map(B,X)^H\ar[r] & \Map(A,X)^H\times_{\Map(A,Y)^H}\Map(B,Y)^H
  }
\end{align*}
is a fibration that is an acyclic fibration if either $j$ or $p$ is a weak equivalence.
\end{prop}

\begin{proof}
Suppose $\function{j}{A}{B}$ is a cofibration and $\function{p}{X}{Y}$ is an acyclic fibration. Let $\function{i}{C}{D}$ be a cofibration in $\C$. We want to verify the pullback corner map satisfies the right lifting property with respect to $i$.
\begin{align}\label{diag:lift_one}
  \xymatrix{
  C\ar[r]\ar[d] & \Map(B,X)^H\ar[d] \\
  D\ar[r]\ar@{.>}[ur] & \Map(A,X)^H\times_{\Map(A,Y)^H}\Map(B,Y)^H.
  }
\end{align}
The solid commutative diagram \eqref{diag:lift_one} in $\C$ has a lift if and only if the solid diagram \eqref{diag:lift_two} in $\C^{G^\op}$has a lift,
\begin{align}\label{diag:lift_two}
  \xymatrix{
  C\cdot(H\backslash G)\ar[r]\ar[d] & \Map(B,X)\ar[d] \\
  D\cdot(H\backslash G)\ar[r]\ar@{.>}[ur] & 
  \Map(A,X)\times_{\Map(A,Y)}\Map(B,Y).
  }
\end{align}
if and only if the solid diagram \eqref{diag:lift_three} in $\C^{G^\op}$ has a lift.
\begin{align}\label{diag:lift_three}
  \xymatrix{
  A\ar[r]\ar[d] & \Map(D\cdot(H\backslash G),X)\ar[d]\\
  B\ar[r]\ar@{.>}[ur] & \Map(C\cdot(H\backslash G),X)
  \times_{\Map(C\cdot(H\backslash G),Y)}
  \Map(D\cdot(H\backslash G),Y).
  }
\end{align}
Hence it is sufficient to verify that the right-hand vertical map in \eqref{diag:lift_three} is an acyclic fibration in $\C$, and hence in $\C^{G^\op}$. The map $\function{i\cdot\id}{C\cdot(H\backslash G)}{D\cdot(H\backslash G)}$ is isomorphic in $\C$ to a coproduct of cofibrations in $\C$, hence is itself a cofibration in $\C$, and the (Enr) axiom \cite[Definition 2.3]{Lewis_Mandell} for monoidal model categories finishes the argument for this case. The other cases are similar.
\end{proof}

\begin{prop}\label{prop:map_fixed_pts}
Let $B$ and $X$ be symmetric sequences and $t\geq 1$. Then for each $s\geq 0$ there is a natural isomorphism in $\C$,
\begin{align*}
  \Map(B[\mathbf{s},\mathbf{t}],X[\mathbf{s}])^{\Sigma_s}\Iso\prod
  \limits_{s_1+\dotsb+s_t=s}
  \Map(B[\mathbf{s_1}]\tensor\dotsb\tensor B[\mathbf{s_t}],X[\mathbf{s}])^
  {\Sigma_{s_1}\times\dotsb\times\Sigma_{s_t}}.
\end{align*}
\end{prop}

\begin{proof}
This follows from the calculation in Proposition~\ref{prop:tensor_calculation}.
\end{proof}

\begin{prop}\label{prop:tensors_are_cofibrations}
Let $G_1,\dotsc,G_n$ be finite groups. 
\begin{itemize}
\item[(a)] Suppose for $k=1,\dotsc,n$ that $\function{j_k}{A_k}{B_k}$ is a cofibration between cofibrant objects in $\C^{G_k^\op}$. Then the induced map 
\begin{align*}
  \function{j_1\tensor\dotsb\tensor j_n}
  {A_1\tensor\dotsb\tensor A_n}{B_1\tensor\dotsb\tensor B_n}
\end{align*}
is a cofibration in $\C^{(G_1\times\dotsb\times G_n)^\op}$ that is an acyclic cofibration if each $j_k$ is a weak equivalence.
\item[(b)] Suppose for $k=1,\dotsc,n$ that $A_k$ is a cofibrant object in $\C^{G_k^\op}$. Then $A_1\tensor\dotsb\tensor A_n$ is a cofibrant object in $\C^{(G_1\times\dotsb\times G_n)^\op}$.
\end{itemize}
\end{prop}

\begin{rem}
By the right-hand adjunction in \eqref{eq:useful_adjunctions_diagrams}, the functor $-\tensor A_k$ preserves initial objects. In particular, if $A_1,\dotsc,A_n$ in the statement of (a) are all initial objects, then $A_1\tensor\dotsb\tensor A_n$ is an initial object in $\C^{(G_1\times\dotsb\times G_n)^\op}$.
\end{rem}

\begin{proof}
For each $n$, statement (b) is a special case of statement (a), hence it is sufficient to verify (a). By induction on $n$, it is enough to verify the case $n=2$. 
Suppose for $k=1,2$ that $\function{j_k}{A_k}{B_k}$ is a cofibration between cofibrant objects in $\C^{G_k^\op}$. The induced map $\function{j_1\tensor j_2}{A_1\tensor A_2}{B_1\tensor B_2}$ factors as
\begin{align*}
  \xymatrix{
  A_1\tensor A_2\ar[r]^{j_1\tensor\id} & 
  B_1\tensor A_2\ar[r]^{\id\tensor j_2} & B_1\tensor B_2,
  }
\end{align*}
hence it is sufficient to verify each of these is a cofibration in $\C^{(G_1\times G_2)^\op}$. Consider any acyclic fibration $\function{p}{X}{Y}$ in $\C^{(G_1\times G_2)^\op}$. We want to show that $j_1\tensor\id$ has the left lifting property with respect to p.
\begin{align*}
\xymatrix{
  A_1\tensor A_2\ar[r]\ar[d] & X\ar[d]\\
  B_1\tensor A_2\ar[r]\ar@{.>}[ur] & Y
}\quad\quad
\xymatrix{ 
  A_1\ar[r]\ar[d] & \Map(A_2,X)^{G_2}\ar[d]\\ 
  B_1\ar[r]\ar@{.>}[ur] & \Map(A_2,Y)^{G_2}
}
\end{align*}
The left-hand solid commutative diagram in $\C^{(G_1\times G_2)^\op}$ has a lift if and only if the right-hand solid diagram in $\C^{G_1^\op}$ has a lift. By assumption $A_2$ is cofibrant in $\C^{G_2^\op}$, hence by Proposition~\ref{prop:corner_map_fixed} the right-hand solid diagram has a lift, finishing the argument that $j_1\tensor\id$ is a cofibration in $\C^{(G_1\times G_2)^\op}$. Similarly, $\id\tensor j_2$ is a cofibration in $\C^{(G_1\times G_2)^\op}$. The case for acyclic cofibrations is similar.
\end{proof}

The following proposition is also useful.

\begin{prop}
Let $G_1$ and $G_2$ be finite groups. Suppose for $k=1,2$ that $\function{j_k}{A_k}{B_k}$ is a cofibration in $\C^{G_k^\op}$. Then the pushout corner map
\begin{align}
\label{eq:naturally_occurring_pushout_corner}
\xymatrix{
  B_1\tensor A_2\amalg_{A_1\tensor A_2} A_1\tensor B_2\ar[r] & 
  B_1\tensor B_2
}
\end{align}
is a cofibration in $\C^{(G_1\times G_2)^\op}$ that is an acyclic cofibration if either $j_1$ or $j_2$ is a weak equivalence.
\end{prop}

\begin{proof}
Suppose for $k=1,2$ that $\function{j_k}{A_k}{B_k}$ is a cofibration in $\C^{G_k^\op}$. Consider any acyclic fibration $\function{p}{X}{Y}$ in $\C^{(G_1\times G_2)^\op}$. We want to show that the pushout corner map \eqref{eq:naturally_occurring_pushout_corner} has the left lifting property with respect to $p$. The solid commutative diagram
\begin{align*}
\xymatrix{
  B_1\tensor A_2\amalg_{A_1\tensor A_2} A_1\tensor B_2\ar[d]\ar[r] & X\ar[d]\\
  B_1\tensor B_2\ar[r]\ar@{.>}[ur] & Y
}\quad\quad
\end{align*}
in $\C^{(G_1\times G_2)^\op}$ has a lift if and only if the solid diagram
\begin{align*}
\xymatrix{ 
  A_1\ar[r]\ar[d] & \Map(B_2,X)^{G_2}\ar[d]\\ 
  B_1\ar[r]\ar@{.>}[ur] & \Map(A_2,X)^{G_2}
  \times_{\Map(A_2,Y)^{G_2}}\Map(B_2,Y)^{G_2}
}
\end{align*}
in $\C^{G_1^\op}$ has a lift. By assumption $A_2\rarrow B_2$ is a cofibration in $\C^{G_2^\op}$, hence Proposition~\ref{prop:corner_map_fixed} finishes the argument that \eqref{eq:naturally_occurring_pushout_corner} is a cofibration in $\C^{(G_1\times G_2)^\op}$. The other cases are similar.
\end{proof}

\subsection{Proofs for the pushout corner map theorems}\label{sec:pushout_corner_proof}

\begin{proof}[Proof of Theorem~\ref{thm:well_behaved_tensor}]
Statements (a) and (b) are equivalent, hence it is sufficient to verify statement (b). Suppose $\function{j}{A}{B}$ is an acyclic cofibration and $\function{p}{X}{Y}$ is a fibration. We want to verify each pullback corner map
\begin{align*}
  \xymatrix{
  \Map^\tensor(B,X)[\mathbf{t}]\ar[r] & \Map^\tensor(A,X)[\mathbf{t}]
  \times_{\Map^\tensor(A,Y)[\mathbf{t}]}
  \Map^\tensor(B,Y)[\mathbf{t}],
  }
\end{align*}
is an acyclic fibration in $\C$. By Proposition~\ref{prop:calc_map_tensor_object} it is sufficient to verify each map 
\begin{align*}
  \xymatrix{
  \Map(B[\mathbf{s}],X[\mathbf{t}\boldsymbol{+}
  \mathbf{s}])^{\Sigma_s}\ar[d] \\
  \Map(A[\mathbf{s}],X[\mathbf{t}\boldsymbol{+}\mathbf{s}])^{\Sigma_s}
  \times_{\Map(A[\mathbf{s}],Y[\mathbf{t}\boldsymbol{+}\mathbf{s}])^{\Sigma_s}}
  \Map(B[\mathbf{s}],Y[\mathbf{t}\boldsymbol{+}\mathbf{s}])^{\Sigma_s}
  }
\end{align*}
is an acyclic fibration in $\C$. Proposition~\ref{prop:corner_map_fixed} completes the proof for this case. The other cases are similar. 
\end{proof}

\begin{proof}[Proof of Theorem~\ref{thm:circ_compatible}]
Statements (a) and (b) are equivalent, hence it is sufficient to verify statement (b). Suppose $\function{j}{A}{B}$ is an acyclic cofibration between cofibrant objects and $\function{p}{X}{Y}$ is a fibration. We want to verify each pullback corner map
\begin{align*}
  \xymatrix{
  \Map^\circ(B,X)[\mathbf{t}]\ar[r] & \Map^\circ(A,X)[\mathbf{t}]
  \times_{\Map^\circ(A,Y)[\mathbf{t}]}
  \Map^\circ(B,Y)[\mathbf{t}],
  }
\end{align*}
is an acyclic fibration in $\C$. If $t=0$, this map is an isomorphism by a calculation in Proposition~\ref{prop:map_seq_calc}. If $t\geq 1$, by Proposition~\ref{prop:calc_map_circ} it is sufficient to show each map 
\begin{align*}
  \xymatrix{
  \Map(B[\mathbf{s},\mathbf{t}],X[\mathbf{s}])^{\Sigma_s}\ar[d] \\
  \Map(A[\mathbf{s},\mathbf{t}],X[\mathbf{s}])^{\Sigma_s}
  \times_{\Map(A[\mathbf{s},\mathbf{t}],Y[\mathbf{s}])^{\Sigma_s}}
  \Map(B[\mathbf{s},\mathbf{t}],Y[\mathbf{s}])^{\Sigma_s}
  }
\end{align*}
is an acyclic fibration in $\C$. By Propositions \ref{prop:map_fixed_pts} and \ref{prop:corner_map_fixed}, it is enough to verify each map
\begin{align*}
  \xymatrix{
  A[\mathbf{s_1}]\tensor\dotsb\tensor A[\mathbf{s_t}]\ar[r] & 
  B[\mathbf{s_1}]\tensor\dotsb\tensor B[\mathbf{s_t}],
  }
\end{align*}
is an acyclic cofibration in $\C^{(\Sigma_{s_1}\times\dotsb\times\Sigma_{s_t})^\op}$. Proposition~\ref{prop:tensors_are_cofibrations} completes the proof for this case. The other cases are similar. 
\end{proof}

\section{Proofs}
\label{sec:proofs_for_modules}

The purpose of this section is to prove Theorems \ref{MainTheorem} and \ref{MainTheorem2}, which establish certain model category structures on algebras and left modules over an operad. In this paper, our primary method of establishing model structures is to use a small object argument. The reader unfamiliar with the small object argument may consult \cite[Section 7.12]{Dwyer_Spalinski} for a useful introduction, after which the (possibly transfinite) versions in \cite{Hirschhorn, Hovey, Schwede_Shipley} should appear quite natural. An account of these techniques is provided in \cite[Section 2]{Schwede_Shipley} which will be sufficient for our purposes. Our proofs of Theorems \ref{MainTheorem} and \ref{MainTheorem2} will reduce to verifying the conditions of Lemma 2.3(1) in \cite{Schwede_Shipley}. This verification amounts to a homotopical analysis of certain pushouts. The reader may contrast this with a path object approach studied in \cite{Berger_Moerdijk}, which amounts to verifying the conditions of Lemma 2.3(2) in \cite{Schwede_Shipley}; compare also \cite{Hinich, Spitzweck}.

\subsection{Arrays and symmetric arrays}
When working with left modules over an operad, we are naturally led to replace $(\C,\tensor,\unit)$ with $(\SymSeq,\tensor,1)$ as the underlying monoidal model category, and hence to working with symmetric arrays.

\begin{defn}\label{def:symmetric_array}
\quad
\begin{itemize}
\item A \emph{symmetric array} in $\C$ is a symmetric sequence in $\SymSeq$; i.e. a functor $\functor{A}{\Sigma^\op}{\SymSeq}$. An \emph{array} in $\C$ is a sequence in $\Seq$; i.e. a functor $\functor{A}{\Omega^\op}{\Seq}$.
\item $\SymArray:=\SymSeq^{\Sigma^\op}\Iso\C^{\Sigma^\op\times\Sigma^\op}$ is the category of symmetric arrays in $\C$ and their natural transformations. $\Array:=\Seq^{\Omega^\op}\Iso\C^{\Omega^\op\times\Omega^\op}$ is the category of arrays in $\C$ and their natural transformations.
\end{itemize}
\end{defn}

It is easy to check that the diagram categories $\SymArray$ and $\Array$ inherit corresponding projective model category structures, where the weak equivalences (resp. fibrations) are the objectwise weak equivalences (resp. objectwise fibrations). Each of these model structures is cofibrantly generated in which the generating cofibrations and acyclic cofibrations have small domains.

Note that all of the statements and constructions which were previously described in terms of $(\C,\tensor,\unit)$ are equally true for $(\SymSeq,\tensor,1)$ and $(\Seq,\hat{\tensor},1)$, and we usually cite and use the appropriate statements and constructions without further comment.

\subsection{Model structures in the $\Sigma$-operad case}
\label{sec:proofs_for_sigma_case}

\begin{proof}[Proof of Theorem~\ref{MainTheorem2}]
We will prove that the model structure on $\LtO$ (resp. $\AlgO$) is created by the adjunction
\begin{align*}
\xymatrix{
  \SymSeq\ar@<0.5ex>[r]^-{\capO\circ-} & \LtO\ar@<0.5ex>[l]^-{U}
}\quad\quad\Bigl(\text{resp.}\quad
\xymatrix{
  \C\ar@<0.5ex>[r]^-{\capO\circ(-)} & \AlgO\ar@<0.5ex>[l]^-{U}
}\Bigr)
\end{align*}
with left adjoint on top and $U$ the forgetful functor. Define a map $f$ in $\LtO$ to be a weak equivalence (resp. fibration) if $U(f)$ is a weak equivalence (resp. fibration) in $\SymSeq$. Similarly, define a map $f$ in $\AlgO$ to be a weak equivalence (resp. fibration) if $U(f)$ is a weak equivalence (resp. fibration) in $\C$. Define a map $f$ in $\LtO$ (resp. $\AlgO$) to be a cofibration if it has the left lifting property with respect to all acyclic fibrations in $\LtO$ (resp. $\AlgO$).

Consider the case of $\LtO$. We want to verify the model category axioms (MC1)-(MC5) in \cite{Dwyer_Spalinski}. By Propositions \ref{prop:colimits_exist} and \ref{prop:limits_exist}, we know that (MC1) is satisfied, and verifying (MC2) and (MC3) is clear. The (possibly transfinite) small object arguments described in the proof of \cite[Lemma 2.3]{Schwede_Shipley} reduce the verification of (MC5) to the verification of Proposition \ref{prop:needed_for_small_object} below. The first part of (MC4) is satisfied by definition, and the second part of (MC4) follows from the usual lifting and retract argument, as described in the proof of \cite[Lemma 2.3]{Schwede_Shipley}. This verifies the model category axioms. By construction, the model category is cofibrantly generated. Argue similarly for the case of $\AlgO$ by considering left $\capO$-modules concentrated at 0, together with Remark~\ref{rem:concentrated_at_zero}.
\end{proof}
 
\begin{rem}
Since the forgetful functors in this proof commute with filtered colimits (Proposition \ref{prop:reflexive_and_filtered_colimits}), the smallness conditions needed for the (possibly transfinite) small object arguments in \cite[Lemma 2.3]{Schwede_Shipley} are satisfied; here, we remind the reader of Basic Assumption \ref{BasicAssumption}.
\end{rem}

\subsection{Analysis of pushouts in the $\Sigma$-operad case}
\label{sec:sigma_operad_case_for_pushouts}
The purpose of this section is to prove the following proposition which we used in the proof of Theorem \ref{MainTheorem2}.

\begin{prop}\label{prop:needed_for_small_object}
Let $\capO$ be a $\Sigma$-operad and $A\in\LtO$. Assume that every object in $\SymArray$ is cofibrant. Consider any pushout diagram in $\LtO$ of the form,
\begin{align}\label{eq:needed_for_small_object}
\xymatrix{
  \capO\circ X\ar[r]^-{f}\ar[d]^{\id\circ i} & A\ar[d]^{j}\\
  \capO\circ Y\ar[r] & A\amalg_{(\capO\circ X)}(\capO\circ Y),
}
\end{align}
such that $\function{i}{X}{Y}$ is an acyclic cofibration in $\SymSeq$. Then $j$
is an acyclic cofibration in $\SymSeq$.
\end{prop}

\begin{rem}\label{rem:concentrated_at_zero}
If $X,Y,A$ are concentrated at $0$, then the pushout diagram \eqref{eq:needed_for_small_object} is concentrated at $0$. To verify this, use Proposition~\ref{prop:useful_calculations_for_circ} and the construction of colimits described in Proposition~\ref{prop:colimits_exist}. 
\end{rem}

A first step in analyzing the pushouts in \eqref{eq:needed_for_small_object} is an analysis of certain coproducts. The following proposition is motivated by a similar argument given in \cite[Section 2.3]{Goerss_Hopkins_moduli} and \cite[Section 13]{Mandell} in the context of algebras over an operad.

\begin{prop}\label{prop:coproduct_modules}
Let $\capO$ be a $\Sigma$-operad, $A\in\LtO$, and $Y\in\SymSeq$. Consider any coproduct in $\LtO$ of the form
\begin{align}\label{eq:coproduct_diagram_modules}
\xymatrix{
  A\amalg(\capO\circ Y).
}
\end{align}  
There exists a symmetric array $\capO_A$ and natural isomorphisms
\begin{align*}
  A\amalg(\capO\circ Y) \Iso \capO_A\circ(Y) = 
  \coprod\limits_{q\geq 0}\capO_A[\mathbf{q}]
    \tensor_{\Sigma_q}Y^{\tensor q}
\end{align*}
in the underlying category $\SymSeq$. If $q\geq 0$, then $\capO_A[\mathbf{q}]$ is naturally isomorphic to a colimit of the form
\begin{align*}
  \capO_A[\mathbf{q}]\Iso
  \colim\biggl(
  \xymatrix{
    \coprod\limits_{p\geq 0}\capO[\mathbf{p}\boldsymbol{+}\mathbf{q}]
    \tensor_{\Sigma_p}A^{\tensor p} & 
    \coprod\limits_{p\geq 0}\capO[\mathbf{p}\boldsymbol{+}\mathbf{q}]
    \tensor_{\Sigma_p}(\capO\circ A)^{\tensor p}\ar@<-0.5ex>[l]^-{d_1}
    \ar@<-1.5ex>[l]_-{d_0}
  }
  \biggl),
\end{align*}
in $\SymSeq$, with $d_0$ induced by operad multiplication and $d_1$ induced by $\function{m}{\capO\circ A}{A}$. 
\end{prop}

\begin{rem}
Other possible notations for $\capO_A$ include $\U_\capO(A)$ or $\U(A)$; these are closer to the notation used in \cite{Elmendorf_Mandell, Mandell} and are not to be confused with the forgetful functors.
\end{rem}

\begin{proof}
It is easy to verify that the coproduct in $\eqref{eq:coproduct_diagram_modules}$ may be calculated by a reflexive coequalizer in $\LtO$ of the form, 
\begin{align*}
  A\amalg(\capO\circ Y)\Iso
  \colim\Bigl(
  \xymatrix{
    (\capO\circ A)\amalg (\capO\circ Y) & 
    (\capO\circ\capO\circ A)\amalg (\capO\circ Y)\ar@<0.5ex>[l]^-{d_1}
    \ar@<-0.5ex>[l]_-{d_0}
  }
  \Bigl).
\end{align*}
The maps $d_0$ and $d_1$ are induced by maps $\function{m}{\capO\circ\capO}{\capO}$ and $\function{m}{\capO\circ A}{A}$, respectively. By Proposition \ref{prop:reflexive_and_filtered_colimits}, this reflexive coequalizer may be calculated in the underlying category $\SymSeq$. There are natural isomorphisms,
\begin{align*}
  (\capO\circ A)\amalg(\capO\circ Y)&\Iso\capO\circ (A\amalg Y)
  \Iso \coprod\limits_{t\geq 0}\capO[\mathbf{t}]\tensor_{\Sigma_t}
  (A\amalg Y)^{\tensor t} \\
  &\Iso \coprod\limits_{q\geq 0}
  \Bigl(
  \coprod\limits_{p\geq 0}\capO[\mathbf{p}\boldsymbol{+}\mathbf{q}]
  \tensor_{\Sigma_p}A^{\tensor p}
  \Bigr)
  \tensor_{\Sigma_q}Y^{\tensor q},
\end{align*}
and similarly,
\begin{align*}
  (\capO\circ\capO\circ A)\amalg(\capO\circ Y)
  \Iso \coprod\limits_{q\geq 0}
  \Bigl(
  \coprod\limits_{p\geq 0}\capO[\mathbf{p}\boldsymbol{+}\mathbf{q}]
  \tensor_{\Sigma_p}(\capO\circ A)^{\tensor p}
  \Bigr)
  \tensor_{\Sigma_q}Y^{\tensor q},
\end{align*}
in the underlying category $\SymSeq$. The maps $d_0$ and $d_1$ similarly factor in the underlying category $\SymSeq$. 
\end{proof}

\begin{rem}
We have used the natural isomorphisms
\begin{align*}
  (A\amalg Y)^{\tensor t}\Iso\coprod\limits_{p+q=t}\Sigma_{p+q}
  \cdot_{\Sigma_p\times\Sigma_q}A^{\tensor p}\tensor Y^{\tensor q},
\end{align*}
in the proof of Proposition \ref{prop:coproduct_modules}. Note that this is just a dressed up form of binomial coefficients.
\end{rem}

\begin{defn}\label{def:filtration_setup_modules}
Let $\function{i}{X}{Y}$ be a morphism in $\SymSeq$ and $t\geq 1$. Define $Q_0^t:=X^{\tensor t}$ and $Q_t^t:=Y^{\tensor t}$. For $0<q<t$ define $Q_q^t$ inductively by the pushout diagrams
\begin{align*}
\xymatrix{
  \Sigma_t\cdot_{\Sigma_{t-q}\times\Sigma_{q}}X^{\tensor(t-q)}
  \tensor Q_{q-1}^q\ar[d]^{i_*}\ar[r]^-{\pr_*} & Q_{q-1}^t\ar[d]\\
  \Sigma_t\cdot_{\Sigma_{t-q}\times\Sigma_{q}}X^{\tensor(t-q)}
  \tensor Y^{\tensor q}\ar[r] & Q_q^t
}
\end{align*}
in $\SymSeq^{\Sigma_t}$. We sometimes denote $Q_q^t$ by $Q_q^t(i)$ to emphasize in the notation the map $\function{i}{X}{Y}$. The maps $\pr_*$ and $i_*$ are the obvious maps induced by $i$ and the appropriate projection maps.
\end{defn}

\begin{rem}
The construction $Q^t_{t-1}$ can be thought of as a $\Sigma_t$-equivariant version of the colimit of a punctured $t$-cube (Proposition~\ref{prop:cube_calculation}). If the category $\C$ is pointed, there is a natural isomorphism $Y^{\tensor t}/Q_{t-1}^t\Iso (Y/X)^{\tensor t}$.
\end{rem}

The following proposition provides a useful description of certain pushouts of left modules, and is motivated by a similar construction given in \cite[Section 12]{Elmendorf_Mandell} in the context of simplicial multifunctors of symmetric spectra.

\begin{prop}\label{prop:small_arg_pushout_modules}
Let $\capO$ be a $\Sigma$-operad, $A\in\LtO$, and $\function{i}{X}{Y}$ in $\SymSeq$. Consider any pushout diagram in $\LtO$ of the form,
\begin{align}\label{eq:small_arg_pushout_modules}
\xymatrix{
  \capO\circ X\ar[r]^-{f}\ar[d]^{\id\circ i} & A\ar[d]^{j}\\
  \capO\circ Y\ar[r] & A\amalg_{(\capO\circ X)}(\capO\circ Y).
}
\end{align}
The pushout in \eqref{eq:small_arg_pushout_modules} is naturally isomorphic to a filtered colimit of the form
\begin{align}\label{eq:filtered_colimit_modules}
  A\amalg_{(\capO\circ X)}(\capO\circ Y)\Iso 
  \colim\Bigl(
  \xymatrix{
    A_0\ar[r]^{j_1} & A_1\ar[r]^{j_2} & A_2\ar[r]^{j_3} & \dotsb
  }
  \Bigr)
\end{align}
in the underlying category $\SymSeq$, with $A_0:=\capO_A[\mathbf{0}]\Iso A$ and $A_t$ defined inductively by pushout diagrams in $\SymSeq$ of the form
\begin{align}\label{eq:good_filtration_modules}
\xymatrix{
  \capO_A[\mathbf{t}]\tensor_{\Sigma_t}Q_{t-1}^t\ar[d]^{\id\tensor_{\Sigma_t}i_*}
  \ar[r]^-{f_*} & A_{t-1}\ar[d]^{j_t}\\
  \capO_A[\mathbf{t}]\tensor_{\Sigma_t}Y^{\tensor t}\ar[r]^-{\xi_t} & A_t
}.
\end{align}
\end{prop}

\begin{proof}
It is easy to verify that the pushout in \eqref{eq:small_arg_pushout_modules} may be calculated by a reflexive coequalizer in $\LtO$ of the form
\begin{align*}
  A\amalg_{(\capO\circ X)}(\capO\circ Y)
  \Iso\colim\Bigl(
  \xymatrix{
    A\amalg(\capO\circ Y)
    & A\amalg(\capO\circ X)\amalg(\capO\circ Y)
    \ar@<-0.5ex>[l]_-{\ol{i}}\ar@<0.5ex>[l]^-{\ol{f}}
  }
  \Bigr).
\end{align*}
By Proposition \ref{prop:reflexive_and_filtered_colimits}, this reflexive coequalizer may be calculated in the underlying category $\SymSeq$. The idea is to reconstruct this coequalizer in $\SymSeq$ via a suitable filtered colimit in $\SymSeq$. A first step is to understand what it means to give a cone in $\SymSeq$ out of this diagram.

The maps $\ol{i}$ and $\ol{f}$ are induced by maps $\id\circ i_*$ and $\id\circ f_*$ which fit into the commutative diagram
\begin{align}
\label{eq:induced_maps_modules}
\xymatrix{
  \capO_A\circ(X\amalg Y)\ar@<-0.5ex>[d]_{\ol{i}}\ar@<0.5ex>[d]^{\ol{f}} & 
  \capO\circ(A\amalg X\amalg Y)\ar[l]
  \ar@<-0.5ex>[d]_{\id\circ i_*}\ar@<0.5ex>[d]^{\id\circ f_*} &
  \capO\circ\bigl((\capO\circ A)\amalg X\amalg Y\bigr)\ar@<-0.5ex>[l]_-{d_0}\ar@<0.5ex>[l]^-{d_1}
  \ar@<-0.5ex>[d]_{\id\circ i_*}\ar@<0.5ex>[d]^{\id\circ f_*}\\
  \capO_A\circ(Y) & \capO\circ(A\amalg Y)\ar[l] & 
  \capO\circ\bigl((\capO\circ A)\amalg Y\bigr)
  \ar@<-0.5ex>[l]_-{d_0}\ar@<0.5ex>[l]^-{d_1}
}
\end{align}
in $\LtO$, with rows reflexive coequalizer diagrams, and maps $i_*$ and $f_*$ in $\SymSeq$ induced by $\function{i}{X}{Y}$ and $\function{f}{X}{A}$ in $\SymSeq$. Here we have used the same notation for both $f$ and its adjoint. By Propositions~\ref{prop:coproduct_modules} and \ref{prop:reflexive_and_filtered_colimits}, the pushout in \eqref{eq:small_arg_pushout_modules} may be calculated by the colimit of the left-hand column of \eqref{eq:induced_maps_modules} in the underlying category $\SymSeq$. By \eqref{eq:induced_maps_modules}, $f$ induces maps $\ol{f}_{q,p}$ which make the diagrams
\begin{align*}
\xymatrix{
  \capO_A\circ(X\amalg Y)\Iso\coprod\limits_{q\geq 0}\coprod\limits_{p\geq 0}
  \Bigl(\ \Bigr)\ar[d]^{\ol{f}} &
  \Bigl(
  \capO_A[\mathbf{p}\boldsymbol{+}\mathbf{q}]\tensor_{\Sigma_p\times\Sigma_q}
  X^{\tensor p}\tensor Y^{\tensor q}
  \Bigr)\ar[l]_-{\inmap_{q,p}}\ar@{.>}[d]^{\ol{f}_{q,p}}\\
  \capO_A\circ(Y)\Iso\coprod\limits_{t\geq 0}\Bigl(\ \Bigr) &
  \Bigl(
  \capO_A[\mathbf{q}]\tensor_{\Sigma_q}Y^{\tensor q}
  \Bigr)\ar[l]_-{\inmap_q}
}
\end{align*}
in $\SymSeq$ commute. Similarly, $i$ induces maps $\ol{i}_{q,p}$ which make the diagrams
\begin{align*}
\xymatrix{
  \capO_A\circ(X\amalg Y)\Iso\coprod\limits_{q\geq 0}\coprod\limits_{p\geq 0}
  \Bigl(\ \Bigr)\ar[d]^{\ol{i}} &
  \Bigl(
  \capO_A[\mathbf{p}\boldsymbol{+}\mathbf{q}]\tensor_{\Sigma_p\times\Sigma_q}
  X^{\tensor p}\tensor Y^{\tensor q}
  \Bigr)\ar[l]_-{\inmap_{q,p}}\ar@{.>}[d]^{\ol{i}_{q,p}}\\
  \capO_A\circ(Y)\Iso\coprod\limits_{t\geq 0}\Bigl(\ \Bigr) &
  \Bigl(
  \capO_A[\mathbf{p}\boldsymbol{+}\mathbf{q}]
  \tensor_{\Sigma_{p+q}}Y^{\tensor (p+q)}
  \Bigr)\ar[l]_-{\inmap_{p+q}}
}
\end{align*}
in $\SymSeq$ commute. We can now describe more explicitly what it means to give a cone in $\SymSeq$ out of the left-hand column of \eqref{eq:induced_maps_modules}. Let $\function{\varphi}{\capO_A\circ(Y)}{\cdot}$ be a morphism in $\SymSeq$ and define $\varphi_q:=\varphi\inmap_q$. Then $\varphi\ol{i}=\varphi\ol{f}$ if and only if the diagrams
\begin{align}
\label{eq:useful_cone_diagrams}
\xymatrix{
  \capO_A[\mathbf{p}\boldsymbol{+}\mathbf{q}]
  \tensor_{\Sigma_p\times\Sigma_q}X^{\tensor p}
  \tensor Y^{\tensor q}\ar[d]^{\ol{i}_{q,p}}\ar[r]^-{\ol{f}_{q,p}} & 
  \capO_A[\mathbf{q}]\tensor_{\Sigma_q}Y^{\tensor q}\ar[d]^{\varphi_q}\\
  \capO_A[\mathbf{p}\boldsymbol{+}\mathbf{q}]
  \tensor_{\Sigma_{p+q}}Y^{\tensor(p+q)}
  \ar[r]^-{\varphi_{p+q}} & \cdot
}
\end{align}
commute for every $p,q\geq 0$. Since $\ol{i}_{q,0}=\id$ and $\ol{f}_{q,0}=\id$, it is sufficient to consider $q\geq 0$ and $p>0$.

The next step is to reconstruct the colimit of the left-hand column of \eqref{eq:induced_maps_modules} in $\SymSeq$ via a suitable filtered colimit in $\SymSeq$. The diagrams \eqref{eq:useful_cone_diagrams} suggest how to proceed. We will describe two filtration constructions that calculate the pushout \eqref{eq:small_arg_pushout_modules} in the underlying category $\SymSeq$. The purpose of presenting the filtration construction \eqref{eq:filtration_modules} is to provide motivation and intuition for the filtration construction \eqref{eq:good_filtration_modules} that we are interested in. Since \eqref{eq:filtration_modules} does not use the gluing construction in Definition \ref{def:filtration_setup_modules} it is simpler to verify that \eqref{eq:filtered_colimit_modules} is satisfied and provides a useful warm-up for working with \eqref{eq:good_filtration_modules}.

Define $A_0:=\capO_A[\mathbf{0}]\Iso A$ and for each $t\geq 1$ define $A_t$ by the pushout diagram
\begin{align}
\label{eq:filtration_modules}
\xymatrix{
  \coprod\limits_{\substack{p+q=t\\ q\geq 0,\, p>0}}
  \capO_A[\mathbf{p}\boldsymbol{+}\mathbf{q}]
  \tensor_{\Sigma_p\times\Sigma_q}X^{\tensor p}
  \tensor Y^{\tensor q}\ar[d]^{i_*}\ar[rr]^-{f_*} && A_{t-1}\ar[d]^{j_t}\\
  \capO_A[\mathbf{t}]\tensor_{\Sigma_t}Y^{\tensor t}\ar[rr]^-{\xi_t} && A_t
}
\end{align}
in $\SymSeq$. The maps $f_*$ and $i_*$ are induced by the appropriate maps $\ol{f}_{q,p}$ and $\ol{i}_{q,p}$. We want to use \eqref{eq:filtration_modules} and \eqref{eq:useful_cone_diagrams} to verify that \eqref{eq:filtered_colimit_modules} is satisfied; it is sufficient to verify the universal property of colimits. By Proposition \ref{prop:coproduct_modules}, the coproduct $A\amalg(\capO\circ Y)$ is naturally isomorphic to a filtered colimit of the form 
\begin{align*}
  A\amalg(\capO\circ Y)\Iso 
  \colim\Bigl(
  \xymatrix{
    B_0\ar[r] & B_1\ar[r] & B_2\ar[r] & \dotsb
  }
  \Bigr)
\end{align*}
in the underlying category $\SymSeq$, with $B_0:=\capO_A[\mathbf{0}]$ and $B_t$ defined inductively by pushout diagrams in $\SymSeq$ of the form
\begin{align*}
\xymatrix{
  {*}\ar[d]\ar[r] & B_{t-1}\ar[d]\\
  \capO_A[\mathbf{t}]\tensor_{\Sigma_t}Y^{\tensor t}\ar[r] & 
  B_t
}
\end{align*}
For each $t\geq 1$, there are naturally occurring maps $B_t\rarrow A_t$, induced by the appropriate $\xi_i$ and $j_i$ maps in \eqref{eq:filtration_modules}, which fit into the commutative diagram 
\begin{align*}
\xymatrix{
  &
  &
  &
  &
  &
  \capO_A\circ(X\amalg Y)
  \ar@<-0.5ex>[d]_{\ol{i}}\ar@<0.5ex>[d]^{\ol{f}}\\
  B_0\ar@{=}[d]\ar[r] & 
  B_1\ar[d]\ar[r] & 
  B_2\ar[d]\ar[r] & 
  \dotsb\ar[r] &
  \colim_t B_t\ar[r]^-{\Iso}\ar[d] & 
  \capO_A\circ(Y)\ar[d]_{\ol{\xi}}\\
  A_0\ar[r]^-{j_1} & 
  A_1\ar[r]^-{j_2} & 
  A_2\ar[r]^-{j_3} & 
  \dotsb\ar[r] &
  \colim_t A_t\ar@{=}[r] &
  \colim_t A_t
}
\end{align*}
in $\SymSeq$; the morphism of filtered diagrams induces a map $\ol{\xi}$. We claim that the right-hand column is a coequalizer diagram in $\SymSeq$. To verify that $\ol{\xi}$ satisfies $\ol{\xi}\,\ol{i}=\ol{\xi}\,\ol{f}$, by \eqref{eq:useful_cone_diagrams} it is enough to check that the diagrams

\begin{align*}
\xymatrix{
  \capO_A[\mathbf{p}\boldsymbol{+}\mathbf{q}]
  \tensor_{\Sigma_p\times\Sigma_q}X^{\tensor p}
  \tensor Y^{\tensor q}\ar[d]^{\ol{i}_{q,p}}\ar[r]^-{\ol{f}_{q,p}} & 
  \capO_A[\mathbf{q}]\tensor_{\Sigma_q}Y^{\tensor q}
  \ar[d]^{\ol{\xi}\,\inmap_q}\\
  \capO_A[\mathbf{p}\boldsymbol{+}\mathbf{q}]
  \tensor_{\Sigma_{p+q}}Y^{\tensor(p+q)}
  \ar[r]^-{\ol{\xi}\,\inmap_{p+q}} & \colim_t A_t
}
\end{align*}
commute for every $q\geq 0$ and $p>0$; this is easily verified using \eqref{eq:filtration_modules}, and is left to the reader. Let $\function{\varphi}{\capO_A\circ(Y)}{\cdot}$ be a morphism in $\SymSeq$ such that $\varphi\ol{i}=\varphi\ol{f}$. We want to verify that there exists a unique map $\function{\ol{\varphi}}{\colim_t A_t}{\cdot}$ in $\SymSeq$ such that $\varphi=\ol{\varphi}\,\ol{\xi}$. Consider the corresponding maps $\varphi_i$ in \eqref{eq:useful_cone_diagrams} and define $\ol{\varphi}_0:=\varphi_0$. For each $t\geq 1$, the maps $\varphi_i$ induce maps $\function{\ol{\varphi}_t}{A_t}{\cdot}$ such that $\ol{\varphi}_{t}\,j_{t}=\ol{\varphi}_{t-1}$ and $\ol{\varphi}_{t}\,\xi_{t}=\varphi_{t}$. In particular, the maps $\ol{\varphi}_t$ induce a map $\function{\ol{\varphi}}{\colim_t A_t}{\cdot}$ in $\SymSeq$. Using \eqref{eq:useful_cone_diagrams} it is an easy exercise (which the reader should verify) that $\ol{\varphi}$ satisfies $\varphi=\ol{\varphi}\,\ol{\xi}$ and that $\ol{\varphi}$ is the unique such map. Hence the filtration construction \eqref{eq:filtration_modules} satisfies \eqref{eq:filtered_colimit_modules}. 

One drawback of \eqref{eq:filtration_modules} is that it may be difficult to analyze homotopically. A hint at how to improve the construction is given by the observation that the collection of maps $\ol{f}_{q,p}$ and $\ol{i}_{q,p}$ satisfy many compatibility relations. The idea is to replace the coproduct in \eqref{eq:filtration_modules}, which is isomorphic to
\begin{align*}
  \capO_A[\mathbf{t}]\tensor_{\Sigma_t}\Bigl[(X\amalg Y)^{\tensor t}
  -Y^{\tensor t}\Bigr],
\end{align*}
with $\capO_A[\mathbf{t}]\tensor_{\Sigma_t}Q_{t-1}^t$, using the gluing construction $Q_{t-1}^t$ in Definition \ref{def:filtration_setup_modules}.  Here, $(X\amalg Y)^{\tensor t}-Y^{\tensor t}$ means the coproduct of all factors in $(X\amalg Y)^{\tensor t}$ except $Y^{\tensor t}$. Define $A_0:=\capO_A[\mathbf{0}]\Iso A$ and for each $t\geq 1$ define $A_t$ by the pushout diagram \eqref{eq:good_filtration_modules} in $\SymSeq$. The maps $f_*$ and $i_*$ are induced by the appropriate maps $\ol{f}_{q,p}$ and $\ol{i}_{q,p}$. Arguing exactly as above for the case of \eqref{eq:filtration_modules}, it is easy to use the diagrams \eqref{eq:useful_cone_diagrams} to verify that \eqref{eq:filtered_colimit_modules} is satisfied. The only difference is that the naturally occurring maps $B_t\rarrow A_t$ are induced by the appropriate $\xi_i$ and $j_i$ maps in \eqref{eq:good_filtration_modules} instead of in \eqref{eq:filtration_modules}.
\end{proof}

The following proposition follows from an argument in \cite[Proof of Lemma 6.2]{Schwede_Shipley}. In an effort to keep the paper relatively self-contained, we have included a proof in Section \ref{sec:punctured_cubes}.

\begin{prop}\label{prop:punctured_cube_calculation}
Let $\function{i}{X}{Y}$ be a cofibration (resp. acyclic cofibration) in $\SymSeq$. Then the induced map $\function{i_*}{Q_{t-1}^t}{Y^{\tensor t}}$ is a cofibration (resp. acyclic cofibration) in the underlying category $\SymSeq$. 
\end{prop}

The following proposition provides a homotopical analysis of the pushout in Proposition \ref{prop:needed_for_small_object}.

\begin{prop}\label{prop:pushouts_for_strong_cofibrancy}
Let $\capO$ be a $\Sigma$-operad and $A\in\LtO$. Assume that $\capO_A$ is cofibrant in $\SymArray$. Consider any pushout diagram in $\LtO$ of the form,
\begin{align}
\xymatrix{
  \capO\circ X\ar[r]^-{f}\ar[d]^{\id\circ i} & A\ar[d]^{j}\\
  \capO\circ Y\ar[r] & A\amalg_{(\capO\circ X)}(\capO\circ Y),
}
\end{align}
such that $\function{i}{X}{Y}$ is a cofibration (resp. acyclic cofibration) in $\SymSeq$. Then each map $\function{j_t}{A_{t-1}}{A_t}$ in the filtration \eqref{eq:filtered_colimit_modules} is a cofibration (resp. acyclic cofibration) in $\SymSeq$. In particular, $j$ is a cofibration (resp. acyclic cofibration) in the underlying category $\SymSeq$.
\end{prop}

\begin{proof}
Suppose $\function{i}{X}{Y}$ is an acyclic cofibration in $\SymSeq$. We want to show that each map $\function{j_t}{A_{t-1}}{A_t}$ is an acyclic cofibration in $\SymSeq$. By the construction of $j_t$ in Proposition~\ref{prop:small_arg_pushout_modules}, it is sufficient to verify each $\id\tensor_{\Sigma_t}i_*$ in \eqref{eq:good_filtration_modules} is an acyclic cofibration. Suppose $\function{p}{C}{D}$ is a fibration in $\SymSeq$. We want to verify $\id\tensor_{\Sigma_t}i_*$ has the left lifting property with respect to $p$. 
\begin{align*}
\xymatrix{
  \capO_A[\mathbf{t}]\tensor_{\Sigma_t}Q_{t-1}^t
  \ar[d]_{\id\tensor_{\Sigma_t}i_*}\ar[r] & C\ar[d]^{p} & 
  \capO_A[\mathbf{t}]\tensor Q_{t-1}^t\ar[d]_{\id\tensor i_*}\ar[r] & C\ar[d]^{p}\\
  \capO_A[\mathbf{t}]\tensor_{\Sigma_t}
  Y^{\tensor t}\ar[r]\ar@{.>}[ur] & D &
  \capO_A[\mathbf{t}]\tensor Y^{\tensor t}\ar[r]\ar@{.>}[ur] & D,
}
\end{align*}
The left-hand solid commutative diagram in $\SymSeq$ has a lift if and only if the right-hand solid diagram in $\SymSeq^{\Sigma_t^\op}$ has a lift. Hence it is sufficient to verify that the solid diagram 
\begin{align*}
\xymatrix{
  \emptyset\ar[r]\ar[d] & \Map^\tensor(Y^{\tensor t}, C)\ar[d]^{(*)}\\
  \capO_A[\mathbf{t}]\ar[r]\ar@{.>}[ur] & \Map^\tensor(Q_{t-1}^t,C)
  \times_{\Map^\tensor(Q_{t-1}^t,D)}
  \Map^\tensor(Y^{\tensor t},D)
}
\end{align*}
in $\SymSeq^{\Sigma_t^\op}$ has a lift. By assumption, $\capO_A$ is a cofibrant symmetric array, hence $\capO_A[\mathbf{t}]$ is cofibrant in $\SymSeq^{\Sigma_t^\op}$ and it is sufficient to verify $(*)$ is an acyclic fibration. By Theorem~\ref{thm:well_behaved_tensor}, it is enough to verify $\function{i_*}{Q_{t-1}^t}{Y^{\tensor t}}$ is an acyclic cofibration in the underlying category $\SymSeq$, and Proposition~\ref{prop:punctured_cube_calculation} finishes the proof. The other case is similar.
\end{proof}

\begin{proof}[Proof of Proposition~\ref{prop:needed_for_small_object}]
By assumption every symmetric array is cofibrant, hence $\capO_A$ is cofibrant, and by Proposition~\ref{prop:pushouts_for_strong_cofibrancy} the map $j$ is an acyclic cofibration in the underlying category $\SymSeq$. 
\end{proof}

\subsection{Punctured cubes}
\label{sec:punctured_cubes}

The purpose of this section is to prove Proposition \ref{prop:punctured_cube_calculation}. 

\begin{defn} Let $t\geq 2$. 
\begin{itemize}
\item $\Cube_t$ is the category with objects the vertices $(v_1,\dotsc,v_t)\in\{0,1\}^t$ of the unit $t$-cube; there is at most one morphism between any two objects, and there is a morphism
\begin{align*}
  (v_1,\dotsc,v_t)\rarrow  (v_1',\dotsc,v_t')
\end{align*}
if and only if $v_i\leq v_i'$ for every $1\leq i\leq t$. In particular, $\Cube_t$ is the category associated to a partial order on the set $\{0,1\}^t$.
\item The \emph{punctured cube} $\pCube_t$ is the full subcategory of $\Cube_t$ with all objects except the terminal object $(1,\dotsc,1)$ of $\Cube_t$.
\end{itemize}
\end{defn}

Let $\function{i}{X}{Y}$ be a morphism in $\SymSeq$. It will be useful to introduce an associated functor $\functor{w}{\pCube_t}{\SymSeq}$ defined objectwise by
\begin{align*}
  w(v_1,\dotsc,v_t):=c_1\tensor\dotsb\tensor c_t\quad\quad\text{with}
  \quad\quad
  c_i :=
  \left\{
    \begin{array}{rl}
    X,&\text{for $v_i=0$,}\\
    Y,&\text{for $v_i=1$,}
  \end{array}
  \right.
\end{align*}
and with morphisms induced by $\function{i}{X}{Y}$. The following proposition is an exercise left to the reader.
\begin{prop}\label{prop:cube_calculation}
Let $\function{i}{X}{Y}$ be a morphism in $\SymSeq$ and $t\geq 2$. There are natural isomorphisms
$
  Q_{t-1}^{t}\Iso
  \colim(w)
$
in $\SymSeq$.
\end{prop}

\begin{proof}[Proof of Proposition~\ref{prop:punctured_cube_calculation}.]
Suppose $\function{i}{X}{Y}$ is an acyclic cofibration in $\SymSeq$. The colimit of the diagram $\functor{w}{\pCube_t}{\SymSeq}$ may be computed inductively using pushout corner maps, and hence by  Proposition~\ref{prop:cube_calculation} there are natural isomorphisms
\begin{align*}
  Q_1^2&\Iso Y\tensor X\amalg_{X\tensor X} X\tensor Y,\\
  Q_2^3&\Iso Y\tensor Y\tensor X\amalg_{\bigl(Y\tensor X\amalg_{X\tensor X} 
  X\tensor Y\bigr)\tensor X}
  \bigl(Y\tensor X\amalg_{X\tensor X} X\tensor Y\bigr)\tensor Y,\ \dotsc
\end{align*}
in the underlying category $\SymSeq$. The same argument provides an inductive construction of the induced map $\function{i_*}{Q_{t-1}^t}{Y^{\tensor t}}$ in the underlying category $\SymSeq$; using the natural isomorphisms in Proposition~\ref{prop:cube_calculation}, for each $t\geq 2$ the $Q_{t-1}^t$ fit into pushout diagrams
\begin{align*}
\xymatrix{
  Q_{t-2}^{t-1}\ar[r]^-{\id\tensor i}
  \ar[d]^{i_*\tensor\id}\tensor X & Q_{t-2}^{t-1}\tensor Y\ar[d]
  \ar@/^1pc/[ddr]^{i_*\tensor\id}\\
  Y^{\tensor(t-1)}\tensor X\ar[r]
  \ar@/_1pc/[drr]^-{\id\tensor i} & 
  Q_{t-1}^t\ar@{.>}[dr]_{\existsunique}^{i_*}\\
  && Y^{\tensor t}
}
\end{align*}
in the underlying category $\SymSeq$ with induced map $\function{i_*}{Q_{t-1}^t}{Y^{\tensor t}}$ the indicated pushout corner map. By iterated applications of Theorem~\ref{thm:well_behaved_tensor}, $i_*$ is an acyclic cofibration in $\SymSeq$. The case for cofibrations is similar.
\end{proof}

\subsection{Model structures in the $\Omega$-operad case}

The purpose of this section is to prove Theorem \ref{MainTheorem}. The strong cofibrancy condition assumed in Theorem~\ref{MainTheorem2} is replaced here by the weaker \emph{monoid axiom} \cite[Definition 3.3]{Schwede_Shipley}, but at the cost of dropping all $\Sigma$-actions; i.e., working with $\Omega$-operads instead of $\Sigma$-operads.

\begin{defn}\label{def:monoid_axiom}
A monoidal model category satisfies the \emph{monoid axiom} if every map which is a (possibly transfinite) composition of pushouts of maps of the form
\begin{align*}
  \function{f\tensor\id}{K\tensor B}{L\tensor B}
\end{align*}
such that $\function{f}{K}{L}$ is an acyclic cofibration and $B\in\C$,
is a weak equivalence.
\end{defn}

\begin{proof}[Proof of Theorem~\ref{MainTheorem}]
We will prove that the model structure on $\LtO$ (resp. $\AlgO$) is created by the adjunction
\begin{align*}
\xymatrix{
  \Seq\ar@<0.5ex>[r]^-{\capO\circhat-} & \LtO\ar@<0.5ex>[l]^-{U}
}\quad\quad\Bigl(\text{resp.}\quad
\xymatrix{
  \C\ar@<0.5ex>[r]^-{\capO\circhat(-)} & \AlgO\ar@<0.5ex>[l]^-{U}
}\Bigr)
\end{align*}
with left adjoint on top and $U$ the forgetful functor. Define a map $f$ in $\LtO$ to be a weak equivalence (resp. fibration) if $U(f)$ is a weak equivalence (resp. fibration) in $\Seq$. Similarly, define a map $f$ in $\AlgO$ to be a weak equivalence (resp. fibration) if $U(f)$ is a weak equivalence (resp. fibration) in $\C$. Define a map $f$ in $\LtO$ (resp. $\AlgO$) to be a cofibration if it has the left lifting property with respect to all acyclic fibrations in $\LtO$ (resp. $\AlgO$).

The model category axioms are verified exactly as in the proof of Theorem \ref{MainTheorem2}; (MC5) is verified by Proposition \ref{prop:needed_for_small_object_omega} below. By construction, the model category is cofibrantly generated.
\end{proof}
 
\subsection{Analysis of pushouts in the $\Omega$-operad case}
The purpose of this section is to prove the following proposition which we used in the proof of Theorem~\ref{MainTheorem}. Since the constructions and arguments are very similar to the $\Sigma$-operad case in Section \ref{sec:sigma_operad_case_for_pushouts}, we only include the constructions and propositions needed for future reference, and state how the arguments differ from the symmetric case.

\begin{prop}\label{prop:needed_for_small_object_omega}
Let $\capO$ be an $\Omega$-operad and $A\in\LtO$. Assume that $(\C,\tensor,\unit)$ satisfies the monoid axiom. Then every (possibly transfinite) composition of pushouts in $\LtO$ of the form
\begin{align}\label{eq:needed_for_small_object_omega}
\xymatrix{
  \capO\circhat X\ar[r]^{f}\ar[d]^{\id\circhat i} & A\ar[d]^{j}\\
  \capO\circhat Y\ar[r] & A\amalg_{(\capO\circhat X)}(\capO\circhat Y),
}
\end{align}
such that $\function{i}{X}{Y}$ is an acyclic cofibration in $\Seq$,
is a weak equivalence in the underlying category $\Seq$.
\end{prop}

\begin{rem}\label{rem:concentrated_at_zero_omega}
If $X,Y,A$ are concentrated at $0$, then the pushout diagram \eqref{eq:needed_for_small_object_omega} is concentrated at $0$. To verify this, argue as in Remark~\ref{rem:concentrated_at_zero}.
\end{rem}

The following is an $\Omega$-operad version of Proposition \ref{prop:coproduct_modules}.

\begin{prop}\label{prop:coproduct_modules_omega}
Let $\capO$ be an $\Omega$-operad, $A\in\LtO$, and $Y\in\Seq$. Consider any coproduct in $\LtO$ of the form
\begin{align}\label{eq:coproduct_diagram_modules_omega}
\xymatrix{
  A\amalg(\capO\circhat Y).
}
\end{align}  
There exists an array $\capO_A$ and natural isomorphisms
\begin{align*}
  A\amalg(\capO\circhat Y) \Iso \capO_A\circhat(Y) = 
  \coprod\limits_{q\geq 0}\capO_A[\mathbf{q}]
    \hat{\tensor} Y^{\hat{\tensor} q}
\end{align*}
in the underlying category $\Seq$. If $q\geq 0$, then $\capO_A[\mathbf{q}]$ is naturally isomorphic to a colimit of the form
\begin{align*}
  \colim\biggl(
  \xymatrix{
    \coprod\limits_{p\geq 0}\capO[\mathbf{p}\boldsymbol{+}\mathbf{q}]
    \tensor
    \Bigl[\frac{\Sigma_{p+q}}{\Sigma_p\times\Sigma_q}\cdot
    A^{\hat{\tensor} p}\Bigr] & 
    \coprod\limits_{p\geq 0}\capO[\mathbf{p}\boldsymbol{+}\mathbf{q}]
    \tensor
    \Bigl[\frac{\Sigma_{p+q}}{\Sigma_p\times\Sigma_q}\cdot
    (\capO\circhat A)^{\hat{\tensor} p}\Bigr]\ar@<-0.5ex>[l]^-{d_1}
    \ar@<-1.5ex>[l]_-{d_0}
  }
  \biggl),
\end{align*}
in $\Seq$, with $d_0$ induced by operad multiplication and $d_1$ induced by $\function{m}{\capO\circhat A}{A}$. 
\end{prop}

\begin{rem}
Other possible notations for $\capO_A$ include $\U_\capO(A)$ or $\U(A)$; these are closer to the notation used in \cite{Elmendorf_Mandell, Mandell} and are not to be confused with the forgetful functors.
\end{rem}

\begin{proof}
It is easy to verify that the coproduct in $\eqref{eq:coproduct_diagram_modules_omega}$ may be calculated by a reflexive coequalizer in $\LtO$ of the form, 
\begin{align*}
  A\amalg(\capO\circhat Y)\Iso
  \colim\Bigl(
  \xymatrix{
    (\capO\circhat A)\amalg (\capO\circhat Y) & 
    (\capO\circhat\capO\circhat A)\amalg (\capO\circhat Y)
    \ar@<0.5ex>[l]^-{d_1}
    \ar@<-0.5ex>[l]_-{d_0}
  }
  \Bigl).
\end{align*}
The maps $d_0$ and $d_1$ are induced by maps $\function{m}{\capO\circhat\capO}{\capO}$ and $\function{m}{\capO\circhat A}{A}$, respectively. By Proposition \ref{prop:reflexive_and_filtered_colimits}, this reflexive coequalizer may be calculated in the underlying category $\Seq$. There are natural isomorphisms,
\begin{align*}
  (\capO\circhat A)\amalg(\capO\circhat Y)&
  \Iso  \coprod\limits_{q\geq 0}
  \Bigl(
  \coprod\limits_{p\geq 0}\capO[\mathbf{p}\boldsymbol{+}\mathbf{q}]
  \tensor
  \Bigl[\frac{\Sigma_{p+q}}{\Sigma_p\times\Sigma_q}\cdot
  A^{\hat{\tensor} p}\Bigr]
  \Bigr)
  \hat{\tensor} Y^{\hat{\tensor} q},
\end{align*}
and similarly,
\begin{align*}
  (\capO\circhat\capO\circhat A)\amalg(\capO\circhat Y)
  \Iso \coprod\limits_{q\geq 0}
  \Bigl(
  \coprod\limits_{p\geq 0}\capO[\mathbf{p}\boldsymbol{+}\mathbf{q}]
  \tensor
  \Bigl[\frac{\Sigma_{p+q}}{\Sigma_p\times\Sigma_q}\cdot
  (\capO\circhat A)^{\hat{\tensor} p}\Bigr]
  \Bigr)
  \hat{\tensor} Y^{\hat{\tensor} q},
\end{align*}
in the underlying category $\Seq$. The maps $d_0$ and $d_1$ similarly factor in the underlying category $\Seq$. It is important to note that the ordering of all tensor power factors is respected, and that we are simply using the symmetric groups in the isomorphisms
\begin{align*}
  (A\amalg Y)^{\hat{\tensor} t}\Iso\coprod\limits_{p+q=t}\Sigma_{p+q}
  \cdot_{\Sigma_p\times\Sigma_q}A^{\hat{\tensor} p}\hat{\tensor} 
  Y^{\hat{\tensor} q}
\end{align*}
to build convenient indexing sets for the tensor powers. Note that this is just a dressed up form of binomial coefficients.
\end{proof}

\begin{defn}\label{def:filtration_setup_modules_omega}
Let $\function{i}{X}{Y}$ be a morphism in $\Seq$ and $t\geq 1$. Define $Q_0^t:=X^{\hat{\tensor} t}$ and $Q_t^t:=Y^{\hat{\tensor} t}$. For $0<q<t$ define $Q_q^t$ inductively by the pushout diagrams
\begin{align*}
\xymatrix{
  \Sigma_t\cdot_{\Sigma_{t-q}\times\Sigma_{q}}X^{\hat{\tensor}(t-q)}
  \hat{\tensor} Q_{q-1}^q\ar[d]^{i_*}\ar[r]^-{\pr_*} & Q_{q-1}^t\ar[d]\\
  \Sigma_t\cdot_{\Sigma_{t-q}\times\Sigma_{q}}X^{\hat{\tensor}(t-q)}
  \hat{\tensor} Y^{\hat{\tensor} q}\ar[r] & Q_q^t
}
\end{align*}
in $\Seq^{\Sigma_t}$. We sometimes denote $Q_q^t$ by $Q_q^t(i)$ to emphasize in the notation the map $\function{i}{X}{Y}$. The maps $\pr_*$ and $i_*$ are the obvious maps induced by $i$ and the appropriate projection maps.
\end{defn}

The following is an $\Omega$-operad version of Proposition \ref{prop:small_arg_pushout_modules}. 

\begin{prop}\label{prop:small_arg_pushout_modules_omega}
Let $\capO$ be an $\Omega$-operad, $A\in\LtO$, and $\function{i}{X}{Y}$ in $\Seq$. Consider any pushout diagram in $\LtO$ of the form,
\begin{align}\label{eq:small_arg_pushout_modules_omega}
\xymatrix{
  \capO\circhat X\ar[r]^-{f}\ar[d]^{\id\circhat i} & A\ar[d]^{j}\\
  \capO\circhat Y\ar[r] & A\amalg_{(\capO\circhat X)}(\capO\circhat Y).
}
\end{align}
The pushout in \eqref{eq:small_arg_pushout_modules_omega} is naturally isomorphic to a filtered colimit of the form
\begin{align}\label{eq:filtered_colimit_modules_omega}
  A\amalg_{(\capO\circhat X)}(\capO\circhat Y)\Iso 
  \colim\Bigl(
  \xymatrix{
    A_0\ar[r]^{j_1} & A_1\ar[r]^{j_2} & A_2\ar[r]^{j_3} & \dotsb
  }
  \Bigr)
\end{align}
in the underlying category $\Seq$, with $A_0:=\capO_A[\mathbf{0}]\Iso A$ and $A_t$ defined inductively by pushout diagrams in $\Seq$ of the form
\begin{align}\label{eq:good_filtration_modules_omega}
\xymatrix{
  \capO_A[\mathbf{t}]\tensorhat Q_{t-1}^t\ar[d]^{\id\tensorhat i_*}
  \ar[r]^-{f_*} & A_{t-1}\ar[d]^{j_t}\\
  \capO_A[\mathbf{t}]\tensorhat Y^{\tensorhat t}\ar[r]^-{\xi_t} & A_t\,.
}
\end{align}
\end{prop}

\begin{proof}
It is easy to verify that the pushout in \eqref{eq:small_arg_pushout_modules_omega} may be calculated by a reflexive coequalizer in $\LtO$ of the form
\begin{align*}
  A\amalg_{(\capO\circhat X)}(\capO\circhat Y)
  \Iso\colim\Bigl(
  \xymatrix{
    A\amalg(\capO\circhat Y)
    & A\amalg(\capO\circhat X)\amalg(\capO\circhat Y)
    \ar@<-0.5ex>[l]_-{\ol{i}}\ar@<0.5ex>[l]^-{\ol{f}}
  }
  \Bigr).
\end{align*}
By Proposition \ref{prop:reflexive_and_filtered_colimits}, this reflexive coequalizer may be calculated in the underlying category $\Seq$. The maps $\ol{i}$ and $\ol{f}$ are induced by maps $\id\circhat i_*$ and $\id\circhat f_*$ which fit into the commutative diagram
\begin{align}\label{eq:induced_maps_modules_omega}
\xymatrix{
  \capO_A\circhat(X\amalg Y)\ar@<-0.5ex>[d]_{\ol{i}}\ar@<0.5ex>[d]^{\ol{f}} & 
  \capO\circhat(A\amalg X\amalg Y)\ar[l]
  \ar@<-0.5ex>[d]_{\id\circhat i_*}\ar@<0.5ex>[d]^{\id\circhat f_*} &
  \capO\circhat\bigl((\capO\circhat A)\amalg X\amalg Y\bigr)
  \ar@<-0.5ex>[l]_-{d_0}\ar@<0.5ex>[l]^-{d_1}
  \ar@<-0.5ex>[d]_{\id\circhat i_*}\ar@<0.5ex>[d]^{\id\circhat f_*}\\
  \capO_A\circhat(Y) & \capO\circhat(A\amalg Y)\ar[l] & 
  \capO\circhat\bigl((\capO\circhat A)\amalg Y)
  \ar@<-0.5ex>[l]_-{d_0}\ar@<0.5ex>[l]^-{d_1}
}
\end{align}
in $\LtO$, with rows reflexive coequalizer diagrams, and maps $i_*$ and $f_*$ in $\Seq$ induced by $\function{i}{X}{Y}$ and $\function{f}{X}{A}$. Here we have used the same notation for both $f$ and its adjoint. By Propositions \ref{prop:coproduct_modules_omega} and \ref{prop:reflexive_and_filtered_colimits}, the pushout in \eqref{eq:small_arg_pushout_modules_omega} may be calculated by the colimit of the left-hand column of \eqref{eq:induced_maps_modules_omega} in the underlying category $\Seq$. We want to reconstruct this colimit via a suitable filtered colimit.

A first step is to describe more explicitly what it means to give a cone in $\Seq$ out of the left-hand column of \eqref{eq:induced_maps_modules_omega}. Let $\function{\varphi}{\capO_A\circhat(Y)}{\cdot}$ be a morphism in $\Seq$ and define $\varphi_q:=\varphi\inmap_q$. Then $\varphi\ol{i}=\varphi\ol{f}$ if and only if the diagrams
\begin{align}
\label{eq:cone_data_omega}
\xymatrix{
  \capO_A[\mathbf{p}\boldsymbol{+}\mathbf{q}]\tensorhat
  \bigl[\Sigma_{p+q}\cdot_{\Sigma_p\times\Sigma_q}X^{\tensorhat p}
  \tensorhat Y^{\tensorhat q}\bigr]\ar[d]^{\ol{i}_{q,p}}
  \ar[r]^-{\ol{f}_{q,p}} & 
  \capO_A[\mathbf{q}]\tensorhat Y^{\tensorhat q}\ar[d]^{\varphi_q}\\
  \capO_A[\mathbf{p}\boldsymbol{+}\mathbf{q}]\tensorhat Y^{\tensorhat(p+q)}
  \ar[r]^-{\varphi_{p+q}} & \cdot
}
\end{align}
commute for every $p,q\geq 0$. Here, the maps $\ol{f}_{q,p}$ and $\ol{i}_{q,p}$ are the obvious maps induced by $f$ and $i$, respectively, exactly as in the proof of Proposition \ref{prop:small_arg_pushout_modules}. Since $\ol{i}_{q,0}=\id$ and $\ol{f}_{q,0}=\id$, it is sufficient to consider $q\geq 0$ and $p>0$.

The next step is to reconstruct the colimit of the left-hand column of \eqref{eq:induced_maps_modules_omega} in $\Seq$ via a suitable filtered colimit in $\Seq$. Define $A_0:=\capO_A[\mathbf{0}]\Iso A$ and for each $t\geq 1$ define $A_t$ by the pushout diagram \eqref{eq:good_filtration_modules_omega} in $\Seq$. The maps $f_*$ and $i_*$ are induced by the appropriate maps $\ol{f}_{q,p}$ and $\ol{i}_{q,p}$. Arguing exactly as in the proof of Proposition \ref{prop:small_arg_pushout_modules}, it is easy to use the diagrams \eqref{eq:cone_data_omega} to verify that \eqref{eq:filtered_colimit_modules_omega} is satisfied.
\end{proof}

The following is a version of Proposition \ref{prop:punctured_cube_calculation} for sequences, and is proved by exactly the same argument.

\begin{prop}\label{prop:induced_map_in_punctured_cube_omega}
Let $\function{i}{X}{Y}$ be a cofibration (resp. acyclic cofibration) in $\Seq$. Then the induced map $\function{i_*}{Q_{t-1}^t}{Y^{\hat{\tensor} t}}$ is a cofibration (resp. acyclic cofibration) in the underlying category $\Seq$.
\end{prop}

\begin{prop}\label{prop:monoid_axiom_for_seq}
Assume that $\C$ satisfies Basic Assumption~\ref{BasicAssumption} and in addition satisfies the monoid axiom. Then $(\Seq,\hat{\tensor},1)$ satisfies the monoid axiom.
\end{prop}

\begin{proof}
Since colimits in $\Seq$ are calculated objectwise, use \eqref{eq:tensors3} together with an argument that the pushout of a coproduct $\amalg_\alpha f_\alpha$ of a finite set of maps can be written as a finite composition of pushouts of the maps $f_\alpha$. 
\end{proof}

\begin{proof}[Proof of Proposition~\ref{prop:needed_for_small_object_omega}]
By Proposition~\ref{prop:small_arg_pushout_modules_omega}, $j$ is a (possibly transfinite) composition of pushouts of maps of the form $\id\tensorhat i_*$, and Propositions~\ref{prop:monoid_axiom_for_seq} and \ref{prop:induced_map_in_punctured_cube_omega} finish the proof.
\end{proof}

\bibliographystyle{plain}
\bibliography{ModulesOperadsMonoidal.bib}

\end{document}